\documentclass[11pt]{article}
\usepackage{amsmath,amsthm,amssymb}
\usepackage{mathtools}
\usepackage[T1]{fontenc}
\usepackage[utf8]{inputenc}
\usepackage[dvipdf]{graphicx}
\usepackage{color}
\usepackage{epstopdf}
\usepackage{dsfont}
\usepackage{enumerate}
\usepackage{enumitem}
\usepackage{bbm}

\definecolor{db}{RGB}{0, 0, 130}
\usepackage[colorlinks=true,citecolor=red,linkcolor=db,urlcolor=blue,pdfstartview=FitH]{hyperref}

\usepackage[normalem]{ulem}
\usepackage[numbers]{natbib}

\definecolor{rp}{rgb}{0.25, 0, 0.75}
\definecolor{dg}{rgb}{0, 0.6, 0}

\newtheorem{theorem}{Theorem}[section]

\newtheorem{definition}{Definition}[section]

\newtheorem{assumption}[theorem]{Assumption}
\newtheorem{lemma}[definition]{Lemma}

\newtheorem{proposition}[definition]{Proposition}
\newtheorem{remark}[definition]{Remark}

\def\1{\mathbbm{1}}
\def\K{\mathbb{K}}
\def\R{\mathbb{R}}

\def\E{\mathbb{E}}
\def\N{\mathbb{N}}

\def\x{\times}
\def\Om{\Omega}
\def\om{\omega}

\def\Cc{\mathcal{C}}
\def\Fc{\mathcal{F}}
\def\F{\mathbb{F}}
\def\P{\mathbb{P}}

\def\gammab{\bar{\gamma}}

\def\Tc{\mathcal{T}}

\def\Lc{\mathcal{L}}
\def\Pc{\mathcal{P}}

\def\Wc{\mathcal{W}}
\def\Mc{\mathcal{M}}
\def\Dc{\mathcal{D}}
\def\Hc{\mathcal{H}}
\def\Xk{X^k}
\def\Xmuk{X^{1,k}}
\def\Xnuk{X^{2,k}}

\def\Wk{W^k}
\def\Qk{Q^k}
\def\<{\langle}
\def\>{\rangle}

\DeclareMathOperator{\Tr}{Tr}
\def\gammab{\bar \gamma}
\def\xr{\mathrm{x}}
\def\Omb{\bar \Om}
\def\omb{\bar \om}
\def\Fbb{\bar \F}
\def\Fcb{\bar{\Fc}}
\def\Pb{\bar \P}
\def\Zb{\bar Z}
\def\mub{\bar \mu}
\def\Omt{\tilde \Om}
\def\omt{\tilde \om}
\def\Ft{\tilde \F}
\def\Fct{\tilde \Fc}
\def\Zt{\tilde Z}

\def\ie{\textit{i.e.}}
\def\eg{\textit{e.g.}}

\title{On McKean--Vlasov Branching Diffusion Processes}

\author{
	Julien Claisse \thanks{Universit\'e Paris-Dauphine, PSL University, CNRS, CEREMADE, Paris. claisse@ceremade.dauphine.fr}
	\and Jiazhi Kang\thanks{Department of Mathematics, The Chinese University of Hong Kong. jzkang@math.cuhk.edu.hk}
        \and Xiaolu Tan\thanks{Department of Mathematics, The Chinese University of Hong Kong. xiaolu.tan@cuhk.edu.hk, Research supported by Hong Kong RGC General Research Fund (projects 14302921).}
}

\date{\today}

\begin{document}

\maketitle

	We study a nonlinear branching diffusion process in the sense of McKean, \ie, where particles are subjected to a mean-field interaction. 
	We consider first a strong formulation of the problem and we provide an existence and uniqueness result  
    by using contraction arguments.
	Then we consider the notion of weak solution and its equivalent martingale problem formulation. In this setting, 
	 we provide a general weak existence result, as well as a propagation of chaos property, \ie,
	the McKean--Vlasov branching diffusion is the limit of a large population branching diffusion process with mean-field interaction.

\section{Introduction}

	The McKean--Vlasov stochastic differential equation (SDE) has been introduced to describe the limit of a large population system,
	where each particle dynamic is ruled by an SDE, whose coefficient depends on an interaction term given by the empirical distribution induced by the whole population.
	In the symmetric setting, where each particle SDE has the same coefficient functions and is driven by an independent Brownian motion,
	the interaction term converges to the distribution of a representative particle as the number of particles tends to infinity.
	Let us mention the pioneering work of McKean \cite{Mckean 1967} and Kac \cite{Kac 1956}, and the pedagogical lecture notes of Snitzman \cite{Sznitman 1991} for the early development of the subject.
	Recently,  McKean--Vlasov SDE have drawn much attention, partially due to the development of the mean-field game (MFG) theory introduced independently and simultaneously by Lasry and Lions \cite{Lasry Lions 2007}, and by Huang, Caines, and Malham\'e \cite{Huang Malhame Caines 2006}.

	\vspace{0.5em}

	While standard McKean--Vlasov SDE describes a large population system, its population size stays unchanged from the beginning to the end.
	In this paper, we will study the McKean--Vlasov SDE in a setting where the population size evolves according to a branching process.
	The mathematical modelling of population dynamic has been intensively studied during more than one century.
	In particular, the branching process theory has been largely developed due to its various applications in biology, ecology, medicine, etc.,
	for which let us simply refer to the book of Bansaye and M\'el\'eard \cite{BansayeMeleard}.
	We are particularly interested in the branching diffusion process, which is a branching process where each particle has a feature, such as the spatial position, whose dynamic is ruled by a diffusion process.
	Such processes have been first studied by Skorokhod \cite{Skorokhod}, Ikeda, Nagasawa and Watanabe \cite{IkedaNW} in 1960s, notably to represent a class of nonlinear PDEs as extension of the Feynman-Kac formula.
	It has also been extended to represent a larger class of semilinear PDEs, and to serve within a Monte Carlo method to solve the corresponding PDEs, see e.g. Henry-Labord\`ere \textit{et al.}~\cite{Henry 2014, LOTTW}.
	Additionally, the control of branching diffusion processes has been studied in a context without mean-field interaction, see e.g. Nisio \cite{Nisio 1985}, Ustunel \cite{Ustunel 1981}, Claisse \cite{Claisse 2018}, Kharroubi and Ocello \cite{KharroubiOcello}, etc.
	We also mention the recent work in Claisse, Ren and Tan \cite{Claisse Ren Tan 2019} where the authors study a controlled branching diffusion process with interaction in a mean-field game setting.

	\vspace{0.5em}

	In this paper, we will study a class of McKean--Vlasov branching diffusion process, where the coefficient functions depend on an interaction term given by the distribution of the process itself.
	Equivalently, it is an extension of the McKean--Vlasov SDE, where the population size dynamic is ruled by a branching process.
	Such processes have already been studied by Fontbona and M\'el\'eard \cite{Fontbona Meleard 2015}, and by Fontbona and Mu\~noz \cite{Fontbona Munoz 2022}, 
	in a setting where the branching process is a birth-death process and the interaction is given by a convolution type term.
	By considering the associated Fokker-Planck equation which describes the dynamic of the marginal distribution, they obtained an existence and uniqueness result as well as a quantitative propagation of chaos result.
	Here we will consider a McKean--Vlasov branching diffusion with more general interaction and in a path-dependent setting.
	In contrast to the Fokker-Planck equation on the marginal distribution of the process as in \cite{Fontbona Meleard 2015}, we will consider the stochastic process itself as  a solution to the McKean--Vlasov SDE with branching.
	As in the classical SDE theory, we will first apply the contraction argument in order to obtain an existence and uniqueness result for strong solutions.
	Next, we consider the notion of weak solution as well as the related martingale problem formulation to obtain a general existence result.
	Moreover, we provide a propagation of chaos result, that is, the McKean--Vlasov branching diffusion describes the limit of a large population branching diffusion process with interaction given by the empirical distribution induced by the whole population.

	\vspace{0.5em}

	The rest of the paper is organized as follows.
	In Section \ref{sec:main_results}, we formulate the McKean--Vlasov branching diffusion and then provide our main results, including strong existence and uniqueness,  weak existence and propagation of chaos.
	In Section \ref{sec:proof_strong}, we provide the proof of strong existence and uniqueness by contraction arguments.
	Finally, the proofs of weak existence as well as a propagation of chaos result by weak convergence technique are completed in Section~\ref{sec:proof_weak}.

\paragraph{Notations}
	Let us first introduce some notations used in the rest of the paper.

	\vspace{0.5em}

	\noindent $\mathrm{(i)}$ Let $(X, \rho)$ be a non-empty metric space, we denote by $\Pc(X)$ (resp. $\Mc(X)$) the space of all Borel probability measures (resp. non-negative finite measures) on $X$.  
	Given a measure $\mu\in\Mc(X)$ and a mapping $f:E\to\R,$ we denote the integration by
	\begin{equation*}
	 \langle\mu,f\rangle := \int_E f(x) \,\mu(dx).
	\end{equation*}
   Further, for $p\geq1$, we denote by $\Pc_p(X)$ the space of probability measures on $X$ with $p$-order moment, \ie,
		$$
			\Pc_p(X):=\left\{\mu\in \Pc(X):\int_{X}\rho(x,x_0)^p \mu(d x)<+\infty\right\},
		$$
		for some (and thus all) fixed point $x_0\in X$.
		Let us equip $\Pc_p(X)$ with the Wasserstein metric $\Wc_p$ given by
		$$
			\Wc_p(\mu,\nu) := \inf_{\lambda \in\Lambda(\mu,\nu)} \left(\int_{X\x X} \rho(x, y)^p\lambda(dx,dy) \right)^{1/p},
		$$
		where $\Lambda(\mu,\nu)$ is the collection of all Borel probability measures $\lambda$ on $X\x X$ whose marginals are $\mu$ and $\nu$ respectively.

		\vspace{0.5em}

		\noindent $\mathrm{(ii)}$ To describe the progeny of the branching process, we use the classical Ulam--Harris--Neveu notation.
		Let
		$$
			\K:= \{\emptyset\}\cup\bigcup_{n=1}^{+\infty}\N^n.
		$$
		Given $k, k'\in\K$ with $k=k_1...k_n$ and $k'=k'_1...k'_m$, we define the concatenation of labels $kk':= k_1...k_nk'_1...k'_m$, and denote by $k\prec k'$ if there exists $\tilde{k}$ such that $k'=k\tilde{k}$. 
		Let us define
		$$
			E 
			~:=~
			\left\{
			\sum_{k\in K} \delta_{(k, x^k)} :
			K \subset \K ~\text{finite},
			x^k\in \R^d,
			\text{for all}~ k, k' \in K, k\nprec k'
			\right\}.
		$$
		Let $\K$ be equipped with the discrete topology, then $E$ is a closed subspace of $\Mc(\K\x\R^d)$ under the weak convergence topology and thus $E$ is a Polish space.
		We provide a metric $d_E$ on $E$ which is consistent with the weak convergence topology:
		for all $e_1, e_2\in E$ such that $e_1 = \sum_{k\in K_1} \delta_{(k, x^k)}$ and $e_2 = \sum_{k\in K_2} \delta_{(k, y^k)}$,
		\begin{equation} \label{eq:def_d_E}
			d_E(e_1,e_2)
			~:=
			\sum_{k\in K_1\cap K_2}                                                                                                                   
			\left(|x^k-y^k|\wedge1\right)
			~+~
			\#(K_1\triangle K_2),
		\end{equation}
		where $K_1\triangle K_2 := (K_1\setminus K_2) \cup (K_2\setminus K_1)$ and $\#(K_1\triangle K_2)$ is the number of element in $K_1\triangle K_2$.
		Let $f:=(f^k)_{k\in\K}:\K\x\R^d\rightarrow \R$ be a function and $e:=\sum_{k\in K}\delta_{(k,x^k)}\in E$, we have
		\begin{align*}
			\<e,f\>
			~=~
			\sum_{k\in K}f^k(x^k).
		\end{align*}
		We also fix the reference point $e_0 \in E,$ the null measure,  which means that the associated set of particle is empty. 
	
		\vspace{0.5em}

		\noindent $\mathrm{(iii)}$ 
		We fix a constant $T > 0$ and an integer $d \ge 1$ throughout the paper,
		and denote by $\Cc^d:=C([0,T],\R^d)$ the space of all $\R^d$-valued continuous paths on $[0,T]$, equipped with the uniform convergence norm $\|\om\| := \sup_{0 \le t \le T} |\om_t|\wedge 1. $
		Next, we denote by $\Dc_E:=D([0,T], E)$ the Skorokhod space of all $E$-valued càdlàg paths on $[0,T]$.
		For the weak formulation, we work with the classical Skorokhod metric on $\Dc_E.$ 
		While, for the strong formulation, we consider the uniform convergence metric on $\Dc_E$  given by 
		\begin{equation} \label{eq:def_d_DE}
			d(\om^1, \om^2) 
			~:=
			\sup_{0\leq t \leq T} d_E(\om^1_t, \om^2_t),
			~\mbox{for all}~
			\om^1, \om^2 \in \Dc_E,
		\end{equation}
        and the truncated metric
        \begin{align*}
            d_t(\om^1, \om^2) 
			~:=
			d(\om^1_{t\wedge\cdot},\om^2_{t\wedge\cdot}),
			~\mbox{for all}~
			\om^1, \om^2 \in \Dc_E
            ~\mbox{and}~
            t\in[0,T],
        \end{align*}
       so that
		\begin{equation}\label{eq:p1}
			\Pc_1(\Dc_E) 
			=
			\Big\{ 
				\mu \in \Pc(\Dc_E) ~:  \int_{\Dc_E} \sup_{t\in[0,T]}\langle\om_t,\1{}\rangle \,\mu(d\om) < \infty 
			\Big\},		
		\end{equation}
		as $d_E(\om_t, e_0) = \langle\om_t,\1{}\rangle$ where $\1: \K\x\R^d \longrightarrow \R$ is the constant function defined by $\1(k,x) = 1$ for all $k \in \K,$ $x\in\R^d.$
		
		Let us also define the stopping operator $\pi_t : \Cc^d \longrightarrow \Cc^d$ (resp. $\pi_t: \Dc_E \longrightarrow \Dc_E$) by $\pi_t(\om):=\om_{t\wedge\cdot}$ for all $\om \in \Cc^d$ (resp. $\om \in \Dc_E$).
		Moreover, given a probability measure $m \in \Pc(\Dc_E)$ and $t \ge 0$, we define the pushforward measure $m_t:= m \circ (\pi_t)^{-1}$. 

\section{Main Results}
\label{sec:main_results}

\subsection{McKean--Vlasov Branching Diffusion}
\label{subsec:prelim}

	Let us  introduce first the McKean--Vlasov branching diffusion process by means of a pathwise description.
	We consider the following coefficient functions:
	$$
		\big( b,\sigma, \gamma, (p_{\ell})_{\ell \in \N} \big)
		:
		[0,T] \x\Cc^d\x\Pc(\Dc_E) \longrightarrow \R^d\x\R^{d\x d} \x  [0, \gammab] \x [0,1]^{\N}, 
	$$
	where $\gammab > 0$ is a fixed constant.
	Namely, $b$ and $\sigma$ are the drift and diffusion coefficient for the movement of each particle, $\gamma$ is the death rate, 
	and $(p_{\ell})_{\ell \in \N}$ is the probability mass function of the progeny distribution. 
	In particular, it holds that $p_{\ell}(\cdot) \in [0,1]$ for each $\ell \in \N$, and $\sum_{\ell \in \N} p_{\ell}(\cdot) = 1$.
	Let us also define a partition $(I_{\ell}(\cdot))_{\ell \in \N}$ of $[0,1]$ by
	$$
		I_{\ell} (\cdot)
		~:=~
		\Big[
			\sum_{i=0}^{\ell -1}p_i (\cdot), ~\sum_{i=0}^{\ell} p_i (\cdot)
		\Big),
		~~\mbox{for each}~
		\ell \in \N.
	$$
	
	Let $(\Om, \Fc, \P)$ be a probability space with filtration $\F = (\Fc_t)_{t \ge 0}$, equipped with a $\Fc_0$-measurable $E$-valued random variable $\xi$,
	and  a family of mutually independent,  $\R^d$-valued Brownian motion $(\Wk)_{k\in\K}$ and Poisson random measures $(\Qk(ds,dz))_{k\in\K}$ on $[0,T] \x [0, \gammab] \x [0,1]$ with Lebesgue intensity measure $ds\, dz$. 	
	
	 We consider first a branching diffusion process with a fixed environment measure $\mu\in \Pc(\Dc_E)$ and initial state $\xi.$ It corresponds to a $E$-valued process $(Z_t)_{t \in [0,T]}$ given by
	\begin{equation} \label{eq:Xk2Z}
		Z_t ~:=~ \sum_{k\in K_t}\delta_{(k,\Xk_t)}, ~~t \in [0,T],
	\end{equation}
	where $K_t$ denotes the collection of all labels of particles alive at time $t \in [0,T]$ and $\Xk_t$ corresponds to the position of particle $k\in K_t.$ In particular, it holds that
	$$
		Z_0 = \xi = \sum_{k \in K_0} \delta_{(k, X^k_0)}.
	$$
	Then, for each $k\in K_t,$ the position $\Xk_t$ evolves as a diffusion process characterized by the following SDE:
	\begin{equation}\label{eq:SDE}
	d\Xk_t = b(t, \Xk_{t\wedge\cdot},\mu_t)dt+\sigma(t, \Xk_{t\wedge\cdot},\mu_t)d\Wk_t,
	\end{equation}
	where $\mu_t$ is the pushforward measure $\mu_t := \mu \circ \pi_t^{-1}.$
	Denote further by $S_k$ the birth time of particle $k.$ In particular, it holds that $S_k = 0$ for each initial particle $k \in K_0$.
	Then each particle $k$ runs a death clock with intensity $\gamma(t, X^k_{t\wedge\cdot}, \mu_t)$, \ie, its death time $T_k$ is given by
	$$
		T_k
		~:=~
		\inf
		\big\{
			t>S_k: \Qk \big(\{t\}\x[0,\gamma(t,\Xk_{t\wedge\cdot},\mu_t)] \x[0,1] \big) =1
		\big\}.
	$$
	Let $U_k$ be the random variable satisfying
	$$
		\Qk \big(\{T_k\} \x [0,\gamma(T_k,\Xk_{T_k\wedge\cdot},\mu_{T_k})] \x \{U_k\} \big) =1,
	$$
	and note that  $U_k$ is uniformly distributed over the interval $[0,1]$.
	In case that $U_k \in I_{\ell}(T_k, X^k_{T_k\wedge\cdot}, \mu_{T_k})$, at time $T_k,$
	the particle $k$ dies and gives birth to $\ell$ offspring particles $\{k1, \cdots, k \ell \}$, so that 
	\begin{align*}
		K_{T_k} = \left(K_{T_k-} \setminus \{k\}\right) \cup \{k1, \cdots,k \ell\}.
	\end{align*}
	In particular, the birth time of the offspring particles corresponds to the death time of the parent particle, \ie, $S_{ki}:= T_k$ for $i = 1, \cdots , \ell$.
	Further, the offspring particles start from the position of the parent particle, \ie,
	$$
		X^{ki}_{S_{ki}} = X^k_{T_k-}, ~~\text{for } i = 1, \cdots, \ell.
	$$
	Moreover,  we define $X^{ki}_t$ as its ancestor position before the birth time, \ie, $X^{ki}_t = X^k_t$ for $t< S_k$ and $i = 1, \cdots , \ell.$

	Then the McKean--Vlasov branching diffusion corresponds to the situation where the environment measure $\mu$ coincides with the law of the branching diffusion process $Z$,  \ie, $\mu = \P\circ Z^{-1},$
	 or equivalently,  
	 $\mu_t = \P \circ (Z_{t \wedge \cdot})^{-1}$ for all  $t\in[0,T].$

	\vspace{0.5em}
	
	Alternatively, let us provide a characterization of the McKean--Vlasov branching diffusion described above through a family of SDEs.
	Denote by $\Lc$ the infinitesimal generator of the diffusion $(b,\sigma)$, \ie, for all $(t, \xr ,m)\in[0,T]\x\Cc^d\x\Pc(\Dc_E)$ and $f\in C^2_b(\R^d,\R)$,
	\begin{align*}
		\Lc f(t, \xr, m)
		~:=~
		\frac{1}{2}\Tr\left(\sigma\sigma^{\top} (t, \xr,m)D^2f( \xr_t)\right) + b(t,\xr,m)\cdot Df( \xr_t).
	\end{align*}
	Then the McKean--Vlasov branching diffusion with initial condition $\xi$ can be characterized as the solution to the following SDE: for all $f=(f^k)_{k\in\K} \in C^2_b (\K \x \R^d, \R),$ $t\in[0,T],$
	\begin{multline} \label{eq:SDE_MKVB}
		\< Z_t, f \>
		= 
		\<\xi ,f\>
		+
		\int_0^t \sum_{k\in K_s}\Lc f^k(s, \Xk_{s\wedge\cdot}, \mu_s)ds
		+
		\int_0^t \sum_{k\in K_s}
		Df^k(\Xk_s)\sigma(s,\Xk_{s\wedge\cdot},\mu_s)
		d\Wk_s \\		
		\int_{(0,t]\x [0,\gammab] \x [0,1]}
		\sum_{k\in K_{s-}}\sum_{ \ell \geq 0}
		\left(\sum_{i=1}^{ \ell} 
		f^{ki}-f^k\right)\left(\Xk_s \right)\1_{[0,\gamma(s,\Xk_{s\wedge\cdot},\mu_s)]\x I_{\ell}(s,\Xk_{s\wedge\cdot},\mu_s)}(z)\Qk(ds,dz),
		\end{multline}
		with the condition
		\begin{equation}\label{eq:SDE_MKVB2}
		\mu_t = \P \circ (Z_{t\wedge \cdot})^{-1}, ~~\text{for all }t \in [0,T].
		\end{equation}

	\begin{remark}
		$\mathrm{(i)}$ In SDE \eqref{eq:SDE_MKVB}, the random processes $X^k_s$ and the random set $K_s$ act together as a part of the $E$-valued process $Z$ via the relation \eqref{eq:Xk2Z}. 
		Conversely, we can represent $X^k_s$ from $Z_s$ by using the relation $X^k_s = \< Z_s, I_k \>$ with $I_k (\tilde{k},x) := x \1_{\{\tilde{k} = k\}}$ for all $(\tilde{k},x) \in \K \x \R^d$.
		
		\vspace{0.5em}
		
		\noindent $\mathrm{(ii)}$ When the environment measure $\mu\in\Pc(\Dc_E)$ is fixed, 
		it becomes a classical branching diffusion process, whose well-posedness is ensured by standard Lipschitz and boundedness conditions, see e.g. Claisse \cite[Proposition 2.1]{Claisse 2018}.
		The crucial point here is therefore the requirement~\eqref{eq:SDE_MKVB2} that $\mu = \Lc(Z).$
		
		\vspace{0.5em}
		
		\noindent $\mathrm{(iii)}$ It might be of interest to consider as in Claisse, Ren and Tan~\cite{Claisse Ren Tan 2019},  an interaction term of the form $(\nu_t)_{t\in[0,T]}, \nu_t\in\Mc(\R^d),$ where for all $f:\R^d\to \R$ measurable bounded,
		\begin{equation*} 
		\int_{\R^d} {f(x) \, \nu_t(dx)} = \E\left[\langle Z_t, f\rangle\right] = \E\Big[\sum_{k\in K_t} {f(X_t^k)}\Big].
		\end{equation*}
		This case is clearly encompassed in our setting. 
		
		\vspace{0.5em}
		
		\noindent $\mathrm{(iv)}$ 	 It might be more natural to consider $\mu_{t-}$ instead of $\mu_t$ in the dynamic described above.  However, although the process $Z$ has jumps in $E$,  the mapping $t \mapsto \mu_t= \P \circ (Z_{t\wedge\cdot})^{-1}$ is continuous on $[0,T].$ 
	 Indeed, since the death time $T_k$ of each particle is generated by the Poisson random measure $Q^k$,
		it holds that  $\P(T_k = t) = 0$ and thus $\P(Z_t = Z_{t-}) = 1$ for all $t \in [0,T].$
	\end{remark}

\subsection{Strong Solution and Well-Posedness}

 Let us first consider the strong formulation of the McKean--Vlasov branching diffusion as defined below.

	\begin{definition}[Strong Solution] \label{def:strong_solution}
		A strong solution to the McKean--Vlasov branching diffusion SDE \eqref{eq:SDE_MKVB}--\eqref{eq:SDE_MKVB2} in a given probability space $(\Om, \Fc, \P)$ w.r.t.\  an initial condition $\xi$ and a fixed family of independent, Brownian motion $(\Wk)_{k\in\K}$ and Poisson random measures $(\Qk)_{k\in\K}$ on $[0,T] \x [0, \gammab] \x [0,1]$ with Lebesgue intensity measure,   is an $E$--valued c\`adl\`ag  process $Z = (Z_t)_{t \in [0,T]}$  adapted to the (augmented) natural filtration of $(\xi,(W_k,Q_k)_{k\in\K})$ satisfying the branching diffusion SDE~\eqref{eq:SDE_MKVB} together with the McKean--Vlasov condition \eqref{eq:SDE_MKVB2}.
	\end{definition}

	Under the following  assumptions, we can show existence and uniqueness of a McKean--Vlasov branching diffusion in the strong sense. 
	Recall that $\Dc_E$ is equipped with the uniform convergence metric defined in \eqref{eq:def_d_DE} so that the Wasserstein distance $\Wc_1$ on $\Pc_1(\Dc_E)$ is defined w.r.t.\ this metric.
	
	\begin{assumption} \label{A.1}
		\noindent $\mathrm{(i)}$ The coefficients functions $(b, \sigma, \gamma, (p_{\ell})_{\ell \ge 0})$ are progressive in the sense that
		$$
			\big(b, \sigma, \gamma, (p_{\ell})_{\ell \ge 0}\big) (t, \xr, m) 
			=
			\big(b, \sigma, \gamma, (p_{\ell})_{\ell \ge 0} \big) (t, \xr_{t \wedge \cdot}, m_t),
		$$
		for all $(t,\xr, m) \in [0,T] \x \Cc^d \x \Pc_1(\Dc_E).$
		
		\vspace{0.5em}	

		\noindent $\mathrm{(ii)}$ The coefficient functions $(b, \sigma, \gamma, (p_{\ell})_{\ell \ge 0})$ are uniformly bounded. The same holds for the function $\sum_{ \ell \ge 0} \ell p_{\ell}$ and we denote $M:=\|\sum_{ \ell \ge 0} \ell p_{\ell}\|_{\infty}<+\infty.$
	\end{assumption}

	\begin{assumption} \label{A.2}
		\noindent $\mathrm{(i)}$ The coefficient functions $b,\sigma,\gamma$ are Lipschitz in $(\xr, m)$ in the sense that, there exists a constant $L>0$ such that
		\begin{align*}
			\big|(b,\sigma,\gamma)(t, \xr ,m) - (b,\sigma,\gamma)(t, \xr',m') \big| 
			~\leq~
			L
			\big( \|\xr_{t\wedge \cdot} - \xr_{t\wedge \cdot}'\| + \Wc_{1}(m_t, m'_t)  \big),
		\end{align*}
		for all $(t, \xr, \xr', m, m') \in [0,T] \x \Cc^d \x \Cc^d \x \Pc_1 (\Dc_E) \x \Pc_1 (\Dc_E)$.

		\vspace{0.5em}	
        
		\noindent $\mathrm{(ii)}$There exist positive constants $(C_{\ell})_{\ell \ge 0}$ such that $M' := \sum_{\ell \in\N} \ell C_{\ell} < \infty$ and 
		\begin{align*}
			\big| p_\ell (t, \xr,m)-p_\ell (t, \xr',m') \big| 
			~\leq~ 
			C_{\ell} \big( \|\xr_{t\wedge \cdot} - \xr'_{t \wedge \cdot} \| + \Wc_{1}(m_t, m'_t) \big),
		\end{align*}
		for all $(t, \xr, \xr', m, m') \in [0,T] \x \Cc^d \x \Cc^d \x \Pc_1 (\Dc_E) \x \Pc_1 (\Dc_E)$ and $\ell \in \N$.
	\end{assumption}

	\begin{remark}
	Let us provide a simple example of a Lipschitz function on $\Pc_1(\Dc_E).$ 
		Let $f = (f^k)_{k \in \K}: \K \x \R^d \to \R$  and assume that there exists $C>0$ such that $|f^k(x)| \le C$ and $|f^k(x)-f^k(y)| \le C|x-y|$  for all $k\in\K,$ $x,y\in\R^d.$
		Define $F: [0,T]\x \Pc_1(\Dc_E) \to \R$ by
		\begin{equation*}
		 F(t,m) 
		 ~:=~
		 \int_{\Dc_E} \langle z(t), f\rangle \,m(dz).
		\end{equation*}
		Then we can easily check that $| F(t,m^1) - F(t,m^2)| \le C \Wc_1(m^1_t,m^2_t)$ for all $m^1, m^2 \in \Pc_1(\Dc_E)$.
		Indeed, given
		$$
			Z^1_t = \sum_{k \in K_t} \delta_{(k, X^{1,k}_t)}
			~~\mbox{and}~~
			Z^2_t = \sum_{k \in K'_t} \delta_{(k, X^{2,k}_t)}
		$$
		two arbitrary branching diffusion processes with distributions $m^1$ and $m^2$ respectively,
		it follows by straightforward computation that
		\begin{align*}
			| F(t, m^1) - F(t, m^2)|
			~=~&
			\Big| \E \Big[ \sum_{k \in K^1_t} f^k(X^{1, k}_t) - \sum_{k \in K^2_t} f^k(X^{2, k}_t) \Big] \Big|
			~\le~
			C \E \big[ d_E( Z^1_t , Z^2_t) \big].
		\end{align*}
	\end{remark}

	\begin{theorem} \label{thm:strong_sol}
		Let Assumptions \ref{A.1} and \ref{A.2} hold and assume further that the initial condition $\xi$ satisfies $\E [ \< \xi, \1\>] = \E [ \# K_0] < \infty.$
		Then there exists a unique strong solution $Z$  to the McKean--Vlasov branching diffusion SDE \eqref{eq:SDE_MKVB}--\eqref{eq:SDE_MKVB2} in the sense of Definition \ref{def:strong_solution} and it satisfies
		\begin{align*}
		\E\left[
			\sup_{0\leq t\leq T} \<Z_t, \1\>
		\right] 
		= \E\left[
			\sup_{0\leq t\leq T} \# K_t
		\right]
		< \infty.
		\end{align*}
	\end{theorem}
	
	The proof of Theorem~\ref{thm:strong_sol} relies on a contraction argument, it is postponed to Section~\ref{sec:proof_strong_sol}. 	 
	Additionally, by using the estimates needed for this contraction argument, we can easily establish a companion stability property for McKean--Vlasov branching diffusion w.r.t.\ the initial condition and the coefficient functions.  See Appendix~\ref{sec:appendix} for more details.

	\begin{remark}
		In contrast to the classical literature on McKean--Vlasov diffusion, where the Lipschitz condition on the coefficients hold w.r.t.\ the metric $\Wc_2$  (see e.g. \cite{Sznitman 1991, Jourdain Meleard Woycznski 2008, Lacker 2018}),
		we perform our analysis with the metric $\Wc_1$ in this setting. 
		We refer to Remark \ref{remark reason for L 1} for a detailed discussion.
	\end{remark}

\subsection{Weak Solution and Propagation of Chaos}

	As in the classical SDE theory,  we can consider the notion of weak solution by letting the probability space as well as the Brownian motions and the Poisson random measures be part of the solution.
	It allows us to establish existence of McKean--Vlasov branching diffusion under relaxed conditions on the coefficients.
	Since there is no fixed probability space, we are given a measure $m_0 \in \Pc(E)$ as initial condition rather than a random variable $\xi$ with distribution $m_0.$

	\begin{definition} \label{def:weak_solution}
		A weak solution to the McKean--Vlasov branching diffusion SDE~\eqref{eq:SDE_MKVB}--\eqref{eq:SDE_MKVB2} with initial condition $m_0 \in \Pc(E)$
		is a term 
		$$	
			\alpha = \left(\Om, \Fc, \F, \P, Z, \mu, (\Wk,\Qk)_{k\in\K} \right)
		$$ 
		satisfying the following conditions:
		
		\vspace{0.5em}
		
		\noindent $\mathrm{(i)}$
		 $(\Om,\Fc,\F,\P)$ is a filtered probability space, equipped with a family of mutually independent,  Brownian motions $(\Wk)_{k\in\K}$ and Poisson random measures $(\Qk)_{k\in\K}$ with Lebesgue intensity measure  on $[0,T] \x [0, \gammab] \x [0,1]$.

		\vspace{0.5em}
		
		\noindent $\mathrm{(ii)}$ $Z$ is a $E$--valued, $\F$--adapted c\`adl\`ag process, such that $\P \circ Z_0^{-1} = m_0$.
			
		\vspace{0.5em}
		
		\noindent $\mathrm{(iii)}$ $\mu$ is a $\Pc(\Dc_E)$--valued random variable independent of $\left(Z_0, (\Wk,\Qk)_{k\in\K}\right)$ such that
			$$
				\mu_t = \Lc(Z_{t\wedge\cdot}\mid\mu_t),
				~~
				t \in [0,T].
			$$

		\vspace{0.5em}
		
		\noindent $\mathrm{(iv)}$ The process $Z_t = \sum_{k \in K_t} \delta_{(k, X^k_t)}$ satisfies the following SDE: for all $f \in C^2_b (\K \x \R^d, \R),$ $t\in [0,T],$
			\begin{multline} \label{eq:branching_MKSDE_weak}
				\< Z_t, f \>
				= 
				\<Z_0 ,f\>
				+
				\int_0^t \sum_{k\in K_s}\Lc f^k(s, \Xk_{s\wedge\cdot}, \mu_s)ds
				+
				\int_0^t \sum_{k\in K_s}
				Df^k(\Xk_s)\sigma(s,\Xk_{s\wedge\cdot},\mu_s)
				d\Wk_s \\
				+
				\int_{(0,t]\x [0,\gammab] \x [0,1]}
				\sum_{k\in K_{s-}}\sum_{ \ell \geq 0}
				\left(\sum_{i=1}^{ \ell} {f^{ki}} - f^k\right)(\Xk_s)\1_{[0,\gamma(s,\Xk_{s\wedge\cdot},\mu_s)]\x I_{\ell}(s,\Xk_{s\wedge\cdot},\mu_s)}(z)\Qk(ds,dz).
			\end{multline}	
	\end{definition}
	
	\begin{remark}
		Notice that $\mu$ is independent of $Z_0$ in Definition \ref{def:weak_solution}. 
		In particular,  $\mu_0 = \Lc(Z_{0 \wedge \cdot})\in\Pc(\Dc(E))$ is deterministic, completely characterized by $m_0\in\Pc(E).$
	\end{remark}
	
		Under the following  assumption, we can show existence of a weak solution to the McKean--Vlasov branching diffusion SDE. 
	Recall that $\Dc_E$ is equipped with the classical Skorokhod metric so that the continuity below has to be understood w.r.t.\ this metric. 
	
	\begin{assumption} \label{A.3}
	 The coefficient functions $(b, \sigma, \gamma, (p_{\ell})_{\ell \ge 0})$ and  $\sum_{ \ell \ge 0} \ell p_{\ell}$ are continuous in $(\xr, m).$
	\end{assumption}
	
	\begin{theorem}\label{thm:weak_existence}
	Let Assumptions \ref{A.1} and \ref{A.3} hold and $m_0\in\Pc_1(E).$ There exists a weak solution to the McKean--Vlasov branching diffusion SDE~\eqref{eq:SDE_MKVB}--\eqref{eq:SDE_MKVB2} in the sense of Definition \ref{def:weak_solution}.
	\end{theorem}
	
	Theorem~\ref{thm:weak_existence} is an immediate consequence of a propagation of chaos property established in Theorem~\ref{thm:weak_existence_limit} below.  
	Namely, we consider a large number $n$ of branching diffusion processes interacting through their empirical distribution in the sense of Definition~\ref{def:weak_solution_n} below,  and we show that, when $n$ goes to infinity, it converges to a McKean--Vlasov branching diffusion in the weak sense.

	\begin{definition} \label{def:weak_solution_n}
		For each $n \ge 1$, a weak notion of $n$--interacting branching diffusion  with initial condition $m_0 \in \Pc(E)$ is a term 
		$$	
			\alpha_n = \left(\Om^n, \Fc^n, \F^n, \P^n,  (Z^{i})_{i = 1, \cdots, n}, (W^{i,k}, Q^{i,k})_{k\in\K, i = 1, \cdots, n} \right),
		$$
		satisfying the following conditions:
		
		\vspace{0.5em}	
		
		\noindent $\mathrm{(i)}$
		$(\Om^n,\Fc^n,\F^n,\P^n)$ is a filtered probability space, equipped with a family of mutually independent,  Brownian motions $(W^{k})_{k\in\K}$ and Poisson random measures $(Q^{k})_{k\in\K}$ with Lebesgue intensity measure on $[0,T] \x [0, \gammab] \x [0,1].$

		\vspace{0.5em}
				
		\noindent $\mathrm{(ii)}$ For each $i=1,\cdots, n$, $Z^{i}$ is a $E$-valued, $\F^n$-adapted c\`adl\`ag process, 
		such that $Z^{1}_0, \cdots, Z^{n}_0$ are i.i.d. with distribution $m_0.$

		\vspace{0.5em}
				
		\noindent $\mathrm{(iii)}$ For each $i=1, \cdots, n,$ the process $Z^{i}_t = \sum_{k \in K^{i}_t} \delta_{(k, X^{i,k}_t)}$ satisfies the following SDE:  for all $f \in C^2_b (\K \x \R^d, \R),$ $t\in[0,T],$
			\begin{multline} \label{eq:n_tree_SDE}
				\< Z^{i}_t, f \>
				= 
				\<Z^{i}_0 ,f\>
				+
				\int_0^t \sum_{k\in K^{i}_s}\Lc f^k(s, X^{i,k}_{s\wedge\cdot}, \mu^n_s)ds
				+
				\int_0^t \!\!\! \sum_{k\in K^{i}_s}
				Df^k(X^{i,k}_s)\sigma(s, X^{i,k}_{s\wedge\cdot},\mu^n_s)
				d W^{i,k}_s\\
				+
				\int_{(0,t]\x [0,\gammab] \x [0,1]} 
				\sum_{k\in K^{i}_{s-}}\sum_{ \ell \geq 0}
				\left(\sum_{j=1}^{ \ell} {f^{kj}} - f^k\right)(X^{i,k}_s)\1_{I_{\ell}(s,X^{i,k}_{s\wedge\cdot},\mu^n_{s-})\x[0,\gamma(s,X^{i,k}_{s\wedge\cdot},\mu^n_{s-})]}(z) Q^{i,k}(ds,dz),
			\end{multline}
			where $\mu^n_t$ corresponds to the empirical measure 
			\begin{equation*}
			\mu^n_t := \frac1n \sum_{i=1}^n \delta_{Z^{i}_{t \wedge \cdot}}.
			\end{equation*}
	\end{definition}

	\begin{theorem} \label{thm:weak_existence_limit}
		Let Assumptions \ref{A.1} and \ref{A.3} hold and $m_0\in\Pc_1(E).$ Then any sequence of distributions $(\P^n \circ (\mu^n)^{-1})_{n \ge 1}$ of $n$--interacting branching diffusion in the sense of Definition~\ref{def:weak_solution_n} admits a converging subsequence in $\Pc( \Pc_1(\Dc_E))$ and the limit identifies as the distribution $\P \circ \mu^{-1}$ of a McKean--Vlasov branching diffusion in the sense of Definition \ref{def:weak_solution}.
		\end{theorem}
		
		The proof of Theorem~\ref{thm:weak_existence_limit} is postponed to Section~\ref{sec:proof_weak_sol}. 
		It relies on the weak convergence of solutions to appropriate martingale problems,  equivalent to the notion of weak solutions.

\begin{remark}
 Under Assumptions \ref{A.1} and \ref{A.2}, there exists a unique strong solution $Z$ and so, by a  Yamada--Watanabe--like theorem, weak uniqueness holds and  Theorem~\ref{thm:weak_existence_limit} reduces to $\mu^n \longrightarrow \P \circ Z^{-1}$ in distribution. 
\end{remark}

\section{Strong Formulation and Well-Posedness}
\label{sec:proof_strong}

We consider in this section the strong formulation of the McKean--Vlasov branching diffusion processes. 
More precisely, we provide the proof of the strong existence and uniqueness result Theorem~\ref{thm:strong_sol} in Section~\ref{sec:proof_strong_sol}.  
To this end,  we start with a series of preliminary technical lemmas in Section~\ref{sec:technical}.
Throughout this section, we consider a fixed probability space $(\Om, \Fc, \P)$ equipped with an initial condition $\xi$ and a family of independent,  Brownian motions $(\Wk)_{k \in \K}$ and Poisson random measures $(\Qk)_{k \in \K}$ on $[0,T] \x [0, \gammab] \x [0,1]$ with Lebesgue intensity measure. We assume that Assumptions \ref{A.1} and \ref{A.2} hold, and further that $\E[\langle\xi,\1{}\rangle] = \E[\# K_0]<+\infty.$

\subsection{Technical Lemmas}\label{sec:technical}

 We start by studying the solution to SDE~\eqref{eq:SDE_MKVB} when the environment measure is fixed. 
	More precisely, assume that we are given a deterministic $\mu \in \Pc_1(\Dc_E)$ and consider $Z_t=\sum_{k\in K_t} {\delta_{(k,\Xk_t)}}$ satisfying for all $ f \in C^2_b (\K \x \R^d, \R),$ $t\in[0,T],$
	\begin{multline} \label{eq:branch_SDE_classical}
		\< Z_t, f \>
		= 
		\<\xi ,f\>
		+
		\int_0^t \sum_{k\in K_s}\Lc f^k(s, \Xk_{s\wedge\cdot}, \mu_s)ds
		+
		\int_0^t \sum_{k\in K_s}
		Df^k(\Xk_s)\sigma(s,\Xk_{s\wedge\cdot},\mu_s)
		d\Wk_s \\
		+
		\int_{(0,t]\x [0,\gammab] \x [0,1]}
		\sum_{k\in K_{s-}}\sum_{ \ell \geq 0}
		\left(\sum_{i=1}^{ \ell} {f^{ki}} - f^k\right)(\Xk_s)\1_{[0,\gamma(s,\Xk_{s\wedge\cdot},\mu_s)]\x I_{\ell}(s,\Xk_{s\wedge\cdot},\mu_s)}(z)\Qk(ds,dz).
	\end{multline}

	By a simple extension of a result from Claisse \cite[Proposition 2.1]{Claisse 2018}, existence and uniqueness of the strong solution holds for the above SDE.
	\begin{lemma} \label{lemma unique existence fixed flow of law}
	 \label{lemma integrability}
		Let $\mu\in\Pc_1(\Dc_E).$
		Then there exists a unique $E$-valued c\`adl\`ag process $Z$ adapted to the (augmented) natural filtration of $(\xi,(W_k,Q_k)_{k\in\K})$  satisfying the branching diffusion SDE~\eqref{eq:branch_SDE_classical}. 
		In addition, it holds
		 \begin{align*}
        \E\left[
            \sum_{k\in\K}\sup_{0\leq t\leq T}\1_{k\in K_t}
        \right]
        \leq
        \E\left[\#K_0\right]e^{\gammab M T}.
    \end{align*}
	\end{lemma}

\begin{proof}
The proof of existence and uniqueness relies on the fact that the Lipschitz assumption on $b$ and $\sigma$ ensures existence and uniqueness of the strong solution to SDE~\eqref{eq:SDE} characterizing the movement of each particle. 
Then we can construct the process step by step by considering the successive branching times. 
It remains to rule out explosion: This follows from the boundedness assumption on  $\gamma$ and $\sum_{\ell\geq0} \ell p_\ell$ which ensures that the population size has finite moments (see below) and thus remains finite.
We refer to~\cite[Proposition 2.1]{Claisse 2018} for more details.

	\vspace{0.5em}

Let us turn now to the moment estimate. Denote by $\bar{N}_t:=\sum_{k\in\K}\sup_{0\leq t\leq T}\1_{k\in K_t}$ the total number of particles having been alive in the time interval $[0,T],$ and observe that
it can be viewed as a modified branching process where particles give birth to new ones without dying. Then
    \begin{align*}
        \bar{N}_t
        =
        \#K_0
        +
        \int_{(0,t]\x[0,\gammab] \x [0,1]}
            \sum_{k\in K_{s-}}\sum_{\ell\geq0}\ell \1_{I_\ell(s,\Xk_{s\wedge\cdot},\mu_s)\x[0,\gamma(s,\Xk_{s\wedge\cdot},\mu_s)]}(z)
        \Qk(ds,dz).
    \end{align*}
    We introduce the stopping times $\tau_n:=\inf\{t\geq0:\bar{N}_t\geq n\}$. It holds
    \begin{align*}
        \E\left[
            \bar{N}_{T\wedge\tau_n}
        \right]
        & = 
        \E\left[\#K_0\right]
        +
        \E\left[
            \int_0^{T\wedge\tau_n}
                \sum_{k\in K_t}\sum_{\ell\geq0}\ell p_\ell(t,\Xk_{t\wedge\cdot},\mu_t)\gamma(t,\Xk_{t\wedge\cdot},\mu_t)
            dt
        \right]\\
        \leq &
        \E\left[\#K_0\right]
        +
        \gammab M \E\left[
            \int_0^{T\wedge\tau_n}\#{K_t}dt
        \right]\\
        \leq &
        \E\left[\#K_0\right]
        +
        \gammab M \E\left[
            \int_0^T \bar{N}_{t\wedge\tau_n}dt
        \right].
    \end{align*}
    By Gr\" onwall's Lemma, we deduce that 
	 \begin{equation*}
        \E\left[
            \bar{N}_{T\wedge\tau_n}
        \right]
        \leq
        \E\left[\#K_0\right]e^{\gammab M T}.
    \end{equation*}   
    The conclusion follows immediately from Fatou's Lemma.
\end{proof}

We now provide a stability result for the branching diffusion solution to SDE~\eqref{eq:branch_SDE_classical} w.r.t.\ the environment measure. It is the key to the contraction argument used in the proof of Theorem~\ref{thm:strong_sol} and  the main technical difficulty.

\begin{lemma}
	\label{lemma compare of two solutions}
	Let $\mu^1,\mu^2\in\Pc_1(\Dc_E).$ Denote by $Z^i$ the solution to SDE \eqref{eq:branch_SDE_classical} given $\mu^i$, $Z^i_0=\xi$ for $i=1,2$ respectively. 	
	Then there exist constants $c_d > 0$, $c_w> 0$ depending on $L$, $\gammab$, $M$ and $M'$ such that for all $T' \in [0,T]$  satisfying $T'+\sqrt{T'}<1/c_d,$ 
	\begin{align*}
		\E\left[
			d_{T'}(Z^1,Z^2)
		\right]
		\leq
		\frac{c_w(T'+\sqrt{T'})}{1-c_d(T'+\sqrt{T'})} \E\left[\# K_0\right]e^{\gammab M T'}\Wc_1(\mu^1_{T'},\mu^2_{T'}).
	\end{align*}
\end{lemma}

\begin{proof}
    Fix some $T'\in[0,T]$ to be determined later. For $t\in[0,T']$, consider the following representations:
    \begin{align*}
        Z^1_t = \sum_{k\in K^1_t}\delta_{(k,\Xmuk_t)}
        ~~\text{and}~~
        Z^2_t = \sum_{k\in K^2_t}\delta_{(k,\Xnuk_t)}.
    \end{align*}
  Let us denote $K^{1\cap2}_t := K^{1}_t \cap K^{2}_t$ and $K^{1\triangle2}_t := K^{1}_t \triangle K^{2}_t.$
    Recall that
    \begin{equation}\label{eq:lem_distance}
        d_{T'}(Z^1,Z^2)
        = \sup_{0\le t\le T'} \left\{\#K^{1\triangle 2}_t +
        \sum_{k\in K^{1\cap2}_t}\left|\Xmuk_t-\Xnuk_t\right|\wedge1\right\}.
    \end{equation}

    \textit{Step 1.}
    We start by estimating the first term in \eqref{eq:lem_distance} counting the number of particles with distinct labels in $Z^1$ and $Z^2.$
For simplicity, we denote for $i = 1,2$, $k\in\K$ and $t\in[0,T]$, $\gamma^{i,k}_t := \gamma(t,X^{i,k}_{t\wedge\cdot},\mu^i_t)$, $I^{i,k}_{\ell,t} := I_\ell(t,X^{i,k}_{t\wedge\cdot},\mu^i_t)$, and $p^{i,k}_{\ell,t}:=p_\ell(t,X^{i,k}_{t\wedge\cdot},\mu^i_t).$
    We begin with the following observation:
    \begin{align}
        \#K^{1\triangle 2}_t\leq J^1_t+J^2_t+J^3_t+J^4_t+J^5_t,
        \label{L}
    \end{align}
    where
    $$
        J^1_t = \int_{(0,t]\x[0,\gammab] \x [0,1]}
            \sum_{k\in K^1_{s-}\setminus K^2_{s-}} \sum_{\ell\geq0} \ell \1_{I^{1,k}_{\ell,s}\x[0,\gamma^{1,k}_s]}(z)
        \Qk(ds,dz),
    $$
    $$
        J^2_t = \int_{(0,t]\x[0,\gammab] \x [0,1]}
            \sum_{k\in K^2_{s-}\setminus K^1_{s-}} \sum_{\ell\geq0} \ell \1_{I^{2,k}_{\ell,s}\x[0,\gamma^{2,k}_s]}(z)
        \Qk(ds,dz),
    $$
    $$
        J^3_t = \int_{(0,t]\x[0,\gammab] \x [0,1]}
            \sum_{k\in K^{1\cap 2}_{s-}}\1_{\gamma^{1,k}_s>\gamma^{2,k}_s} \sum_{\ell\geq0} (\ell+1)\1_{I^{1,k}_{\ell,s}\x (\gamma^{2,k}_s, \gamma^{1,k}_s]}(z)
        \Qk(ds,dz),
    $$
    $$
        J^4_t = \int_{(0,t]\x[0,\gammab] \x [0,1]}
            \sum_{k\in K^{1\cap 2}_{s-}}\1_{\gamma^{1,k}_s<\gamma^{2,k}_s} \sum_{\ell\geq0} (\ell+1)\1_{I^{2,k}_{\ell,s}\x (\gamma^{1,k}_s, \gamma^{2,k}_s]}(z)
        \Qk(ds,dz),
    $$
    and
    $$
        J^5_t = \int_{(0,t]\x[0,\gammab] \x [0,1]}
            \sum_{k\in K^{1\cap 2}_{s-}} \sum_{\ell,\ell '\geq0}\left|\ell-\ell '\right| \1_{I^{1,k}_{\ell,s}\cap I^{2,k}_{\ell',s}\x [0,\gamma^{1,k}_s\wedge\gamma^{2,k}_s]}(z)
        \Qk(ds,dz).
    $$
     To obtain the above estimation, we distinguish between particles born from a parent in the common set $K^{1\cap2}_t$ or not. In the latter case, we assume that all newly born particles also lie in the set $K^{1\triangle 2}_t$. Then $J^1_t$ and $J^2_t$ correspond to the case when parents are from $K^1_t\setminus K^2_t$ and $K^2_t\setminus K^1_t$ respectively.  In the former case, when parents are from $K^{1\cap2}_t,$ the terms $J^3_t$ and $J^4_t$ take care of the case when a particle in one tree branches while the other particle sharing the same label does not. While the last term $J_t^5$  concerns particles in two trees sharing a common label branching simultaneously, but giving birth to different number of progeny.
    
    Let us deal with each term in~\eqref{L} separately. Firstly, we have
    \begin{align*}
        \E\left[
            \sup_{0\leq t\leq T'} J^1_t 
        \right]
       & = 
        \E\left[
            \int_{(0,T']\x[0,\gammab] \x [0,1]}
                \sum_{k\in K^1_{t-}\setminus K^2_{t-}} \sum_{\ell\geq 0} \ell \1_{I^{1,k}_{\ell,t}\x[0,\gamma^{1,k}_t]}(z)
            \Qk(dt,dz)
        \right]\\
        & = 
        \E\left[
            \int_0^{T'}
                \sum_{k\in K^1_t\setminus K^2_t}\gamma^{1,k}_t\sum_{\ell\geq 0} \ell p^{1,k}_{\ell,t}
            dt
        \right]\\
        & \leq
        \gammab M\E\left[
            \int_0^{T'}
                \# (K^1_t\setminus K^2_t)
            dt
        \right].
    \end{align*}
    Similarly, we have 
    \begin{equation*}
    \E\left[\sup_{0\leq t\leq T'}J^2_t\right]\leq\gammab M\E\left[\int_0^{T'}\#(K^2_t\setminus K^1_t)dt\right].
    \end{equation*}
    We deduce that
    \begin{align}\label{J1J2}
        \E\left[
            \sup_{0\leq t\leq T'} \left\{J^1_t + J^2_t\right\}
        \right]
        \leq
        \gammab M\E\left[
            \int_0^{T'}\#K^{1\triangle 2}_tdt
        \right].
    \end{align}
    In addition, using the same approach, we derive that
    \begin{align}
        \label{regularity J3J4}
        \E\left[
            \sup_{0\leq t\leq T'}\left\{J^3_t+J^4_t\right\}
        \right]
        \leq 
        (M+1)\E\left[
            \int_0^{T'}
                \sum_{k\in K^{1\cap2}_t}
                \left|
                    \gamma^{1,k}_t-\gamma^{2,k}_t
                \right|
            dt
        \right].
    \end{align}
    By the Lipschitz condition on $\gamma$ from Assumption \ref{A.2}, it follows that
    \begin{equation}
        \E\left[
            \sup_{0\leq t\leq T'}\left\{J^3_t+J^4_t\right\}
        \right]        
         \leq
        L(M+1)\E\left[
            \int_0^{T'}
                \sum_{k\in K^{1\cap2}_t}
                \left(
                    \left\|\Xmuk_{t\wedge\cdot}-\Xnuk_{t\wedge\cdot}\right\|
                    +
                    \Wc_1(\mu^1_t,\mu^2_t)
                \right)
            dt
        \right].\label{J3J4}
    \end{equation}
    Finally, using Lemma \ref{lemma wasserstein distance of p} below,  we obtain that 
    \begin{align}
        \label{regularity J5}
        \E\left[
            \sup_{0\leq t\leq T'}J^5_t
        \right]
        & \leq
        \gammab\E\left[
            \int_0^{T'}
                \sum_{k\in K^{1\cap2}_t}\int_0^1
                    \sum_{\ell,\ell'\geq0}\left|\ell-\ell'\right|\1_{I^{1,k}_{\ell,t}\cap I^{2,k}_{\ell',t}}(u)
                du
            dt
        \right] 
        \nonumber \\
        & \leq 
           \gammab \E\left[
            \int_0^{T'}
                \sum_{k\in K^{1\cap2}_t}
                \sum_{\ell\ge 0} \ell \left|
                    p^{1,k}_{\ell,t}-p^{2,k}_{\ell,t}
                \right|
            dt
            \right].
    \end{align}
    Thus, by the Lipschitz assumption on $(p_\ell)_{\ell\in\N}$ from Assumption \ref{A.2}, we deduce as above that
    \begin{equation}
        \E\left[
            \sup_{0\leq t\leq T'}J^5_t
        \right]
        \leq
        \gammab  M'\E\left[
            \int_0^{T'}
                \sum_{k\in K^{1\cap2}_t}
                \left(
                    \left\|\Xmuk_{t\wedge\cdot}-\Xnuk_{t\wedge\cdot}\right\|
                    +
                    \Wc_1(\mu^1_t,\mu^2_t)
                \right)
            dt
        \right]\label{J5}
    \end{equation}
    We finish this step by combining \eqref{L}, \eqref{J1J2}, \eqref{J3J4} and \eqref{J5} to obtain
     \begin{equation}
        \E\left[
            \sup_{0\leq t\leq T'}\#K^{1\triangle 2}_t
        \right]
        \leq 
        C \E\left[
            \int_0^{T'} \bigg( \#K^{1\triangle 2}_t +
             \sum_{k\in K^{1\cap2}_t}
                \left(\left\|\Xmuk_{t\wedge\cdot}-\Xnuk_{t\wedge\cdot}\right\|+\Wc_1(\mu^1_t,\mu^2_t)\right)\bigg)dt
        \right],\label{estimate for branching}
    \end{equation}
    where the constant $C>0$ depends on $(\gammab,L,M,M').$

    \textit{Step 2.}
	Now we turn to the second term in \eqref{eq:lem_distance} measuring the distance between positions of particles with common labels in $Z_t^1$ and $Z_t^2.$     For simplicity, we denote for $i = 1,2$, $k\in\K$ and $t\in[0,T]$, $b^{i,k}_t := b(t,X^{i,k}_{t\wedge\cdot},\mu^i_t)$ and $\sigma^{i,k}_t := \sigma(t,X^{i,k}_{t\wedge\cdot},\mu^i_t).$ We observe first that
for any $k\in K^{1\cap2}_t,$
    \begin{equation*}
        \Xmuk_t-\Xnuk_t
        =
        \Xmuk_{\bar{S}^k} - \Xnuk_{\bar{S}^k}\\
         +
            \int_0^t
                \1_{k\in K^{1\cap2}_s}\left(b^{1,k}_s - b^{2,k}_s\right)
            ds 
         +
            \int_0^t
                \1_{k\in K^{1\cap2}_s}\left(\sigma^{1,k}_s - \sigma^{2,k}_s\right)
            d\Wk_s,
    \end{equation*}
    where $\bar{S}^k=\max(S^{1,k},S^{2,k})$ is the last time of birth of particle $k,$ so that $k\in K^{1\cap2}_s$ whenever $s\in[\bar{S}^k,t].$  
    It follows that
    \begin{multline} \label{eq:ineq_path}
    \1_{k\in K^{1\cap2}_t} \left\|\Xmuk_{t\wedge\cdot}-\Xnuk_{t\wedge\cdot}\right\| \le \1_{k\in K^{1\cap2}_t} \left\|\Xmuk_{\bar{S}^k\wedge\cdot}-\Xnuk_{\bar{S}^k\wedge\cdot}\right\| \\
    + 
    \int_0^t
                \1_{k\in K^{1\cap2}_s}\left| b^{1,k}_s - b^{2,k}_s \right|
            ds
         +
           \sup_{0\le r\le t} \left|\int_0^r
                \1_{k\in K^{1\cap2}_s}\left(\sigma^{1,k}_s - \sigma^{2,k}_s\right)
            d\Wk_s\right|.
    \end{multline}
       Let us deal first with the last two terms on the r.h.s.\ of~\eqref{eq:ineq_path}. By using the Lipschitz condition on $b$ and $\sigma$ from Assumption~\ref{A.2},  we have for all $k\in\K,$
    \begin{equation*}
        \E\left[
            \int_0^{t}
                \1_{k\in K^{1\cap2}_s}
                \left|
                    b^{1,k}_s - b^{2,k}_s
                \right|
            ds
        \right] 
        \leq 
        L\E\left[
            \int_0^{t}
                \1_{k\in K^{1\cap2}_s} 
                \left(
                    \left\|\Xmuk_{s\wedge\cdot}-\Xnuk_{s\wedge\cdot}\right\|
                    +
                    \Wc_1(\mu^1_s,\mu^2_s)
                \right)
            ds
        \right],
    \end{equation*}
    as well as 
        \begin{align*}
        \E\left[
            \sup_{0\leq r\leq t}
            \left|
                \int_0^r
                    \1_{k\in K^{1\cap2}_s}
                    \left(
                       \sigma^{1,k}_s - \sigma^{2,k}_s
                    \right)
                d\Wk_s
            \right|
        \right]
        &  \leq
        C_1\E\left[
            \left(
                \int_0^{t}
                    \1_{k\in K^{1\cap2}_s}
                    \left|
                       \sigma^{1,k}_s - \sigma^{2,k}_s
                    \right|^2
                ds
            \right)^{\frac{1}{2}}
        \right]\\
         & \leq
        C_1\sqrt{t}\E\left[
            \sup_{0\leq s\leq t}
            \1_{k\in K^{1\cap2}_s}
            \left|
                \sigma^{1,k}_s - \sigma^{2,k}_s
            \right|
        \right] \\
        & \leq
        C_1 L \sqrt{t} \E\left[
            \sup_{0\leq s\leq t}\1_{k\in K^{1\cap2}_s} 
            \left(
            \left\|\Xmuk_{s\wedge\cdot}-\Xnuk_{s\wedge\cdot}\right\|
                    +
            \Wc_1(\mu^1_s,\mu^2_s)
            \right)
        \right],
    \end{align*}    
    where the first inequality follows from the Burkholder--Davis--Gundy inequality. 
    Regarding the first term on the r.h.s.\  of~\eqref{eq:ineq_path}, we distinguish between particles whose parents are in the common set or not. In the latter case, we use the fact that $\|\Xmuk_{\bar{S}^k\wedge\cdot}-\Xnuk_{\bar{S}^k\wedge\cdot}\|\le 1$ so that 
    \begin{multline*}
    \E\left[\sum_{k\in \K}\sup_{0\le t \le T'} \1_{k\in K^{1\cap2}_t, k^-\in K^{1\triangle 2}_{\bar{S}^k-}} \left\|\Xmuk_{\bar{S}^k\wedge\cdot}-\Xnuk_{\bar{S}^k\wedge\cdot}\right\|\right] \\ 
    \le \E\left[ \sum_{k\in \K}\sup_{0\le t \le T'} \1_{k\in K^{1}_t\cup K^{2}_t, k^-\in K^{1\triangle 2}_{\bar{S}^k-}}\right]
    = \E\left[J^1_{T'} + J_{T'}^2\right] 
    \le \gammab M\E\left[
            \int_0^{T'}\#K^{1\triangle 2}_tdt
        \right],
    \end{multline*}
    where $k^-$ denotes the parent of particle $k$ and the last inequality follows from~\eqref{J1J2}.  
    Note that $\sum_{k\in \K}\sup_{0\le t \le T'} \1_{k\in K^{1}_t\cup K^{2}_t, k^-\in K^{1\triangle 2}_{\bar{S}^k-}}$ corresponds to the total number of particles born from a parent in the distinct set $K^{1\triangle 2}$ before time $T'.$
    Alternatively, when the parent belongs to the common set, it holds 
    \begin{multline*}
     \E\left[ \sum_{k\in \K}\sup_{0\le t \le T'} \1_{k\in K^{1\cap2}_t, k^-\in K^{1\cap 2}_{\bar{S}^k-}} \left\|\Xmuk_{\bar{S}^k\wedge\cdot}-\Xnuk_{\bar{S}^k\wedge\cdot}\right\|\right] \\
        \begin{aligned}
        & = \E\left[
            \int_{(0,T']\x [0,\gammab]\x [0,1]}
                \sum_{k\in K^{1\cap2}_{t-}} \sum_{\ell,\ell'\geq0}(\ell\wedge \ell')\1_{I^{1,k}_{\ell,t}\cap I^{2,k}_{\ell',t}\x[0,\gamma^{1,k}_t\wedge\gamma^{2,k}_t]}(z) \left\|\Xmuk_{t\wedge\cdot}-\Xnuk_{t\wedge\cdot}\right\|
            \Qk(dt,dz)
        \right]\\
        & \leq 
        \gammab\E\left[
            \int_0^{T'}\sum_{k\in K^{1\cap2}_t}
                \sum_{\ell,\ell'\geq0}\ell \,\big|I^{1,k}_{\ell,t}\cap I^{2,k}_{\ell',t}\big| \left\|\Xmuk_{t\wedge\cdot}-\Xnuk_{t\wedge\cdot}\right\|
            dt
        \right]\\
        & = 
        \gammab\E\left[
            \int_0^{T'}\sum_{k\in K^{1\cap2}_t}
                \sum_{\ell\geq0}\ell p^{1,k}_{\ell,t} \left\|\Xmuk_{t\wedge\cdot}-\Xnuk_{t\wedge\cdot}\right\|
            dt
        \right]\nonumber\\
        & \leq
        \gammab M \E\left[ \int_0^{T'}
            \sum_{k\in K^{1\cap2}_t}
                    \left\|\Xmuk_{t\wedge\cdot}-\Xnuk_{t\wedge\cdot}\right\|
        \right].
        \end{aligned}
    \end{multline*}
Taking successively supremum over $t\in[0,T'],$ sum over $k\in \K$ and expectation in \eqref{eq:ineq_path}, we conclude by combining the inequalities above that
    \begin{multline}
        \E\left[\sum_{k\in \K}\sup_{0\leq t\leq T'}\1_{k\in K^{1\cap2}_t}\left\|\Xmuk_{t\wedge\cdot}-\Xnuk_{t\wedge\cdot}\right\|\right] \\
        \le        
        C\sqrt{T'}\E\left[\sum_{k\in \K}\sup_{0\leq t\leq T'}\1_{k\in K^{1\cap2}_t} \left(
            \left\|\Xmuk_{t\wedge\cdot}-\Xnuk_{t\wedge\cdot}\right\|
                    +
            \Wc_1(\mu^1_t,\mu^2_t)
            \right)\right] \\
        +
        C \E\left[
            \int_0^{T'}\bigg(\#K^{1\triangle 2}_t +
             \sum_{k\in K^{1\cap2}_t}
                \left(\left\|\Xmuk_{t\wedge\cdot}-\Xnuk_{t\wedge\cdot}\right\|+\Wc_1(\mu^1_t,\mu^2_t)\right)\bigg)dt
        \right], \label{estimate for movement}
    \end{multline}
    where the constant $C>0$ depends on $(\gammab,L,M).$

	\vspace{0.5em}

    \textit{Step 3.}
    Let us denote 
     \begin{equation*}
       \bar{d}_{T'}(Z^1,Z^2)
        := \sup_{0\leq t\leq T'}\#K^{1\triangle 2}_t
           +
            \sum_{k\in \K}\sup_{0\leq t\leq T'}\1_{k\in K^{1\cap2}_t}\left\|\Xmuk_{t\wedge\cdot}-\Xnuk_{t\wedge\cdot}\right\|.
    \end{equation*}
    Observe that 
    \begin{equation*}
     \E\left[
            \int_0^{T'} \bigg(\#K^{1\triangle 2}_t +
             \sum_{k\in K^{1\cap2}_t}
                \left\|\Xmuk_{t\wedge\cdot}-\Xnuk_{t\wedge\cdot}\right\| \bigg) dt
          \right] 
          \le 
          T'
           \E\left[
           \bar{d}_{T'}(Z^1,Z^2)
        \right],
    \end{equation*}
    and 
    \begin{align*}
    \E\left[
            \int_0^{T'}
             \sum_{k\in K^{1\cap2}_t}
                \Wc_1(\mu^1_t,\mu^2_t) 
                \,dt
        \right] 
       & \le 
        T'
            \E\left[
             \sum_{k\in\K} \sup_{0\leq t\leq T'}\1_{k\in K^{1\cap2}_t}  \Wc_1(\mu^1_{t},\mu^2_{t}) 
        \right] \\
        & \le 
            T' \E\left[\#K_0\right]e^{\gammab M T'}
            \Wc_1(\mu^1_{T'},\mu^2_{T'}),
    \end{align*}    
    where the last inequality follows from Lemma \ref{lemma integrability} and $\Wc_1(\mu^1_{t},\mu^2_{t}) \le \Wc_1(\mu^1_{T'},\mu^2_{T'})$ for $t\in[0,T'].$    
    Combining \eqref{estimate for branching} and \eqref{estimate for movement} and using both inequalities above,
   we can find some positive constants $c_d>0$ and $c_w>0$ depending on $(\gammab,L,M,M')$ such that
    \begin{equation*}
       \E\left[
           \bar{d}_{T'}(Z^1,Z^2)
        \right] 
        \le c_d(T'+\sqrt{T'})
       \E\left[
           \bar{d}_{T'}(Z^1,Z^2)
        \right]
        + c_w(T'+\sqrt{T'}) \E\left[\#K_0\right]e^{\gammab M T'} \Wc_1(\mu^1_{T'},\mu^2_{T'})   .
    \end{equation*}
    Choosing $T'>0$ small enough to have $1-c_d(T'+\sqrt{T'})>0,$ we deduce that
    \begin{equation*}
       \E\left[
           \bar{d}_{T'}(Z^1,Z^2)
        \right]
         \leq 
        \frac{c_w(T'+\sqrt{T'})}{1-c_d(T'+\sqrt{T'})} \E\left[\#K_0\right]e^{\gammab M T'}\Wc_1(\mu^1_{T'},\mu^2_{T'}).
    \end{equation*}
    The conclusion follows immediately by observing that $d_{T'}(Z^1,Z^2)\le \bar{d}_{T'}(Z^1,Z^2).$
\end{proof}

\begin{lemma}
    \label{lemma wasserstein distance of p}
    Let Assumptions \ref{A.1} hold. For any $(t,\xr,\xr',m,m')\in [0,T]\x\Cc^d\x\Cc^d\x\Pc_1(\Dc_E)\x\Pc_1(\Dc_E)$, we have
    \begin{equation*}
        \int_0^1
            \sum_{\ell,\ell'\geq0}\left|\ell-\ell'\right|\1_{I_\ell(t,\xr,m)\cap I_{\ell'}(t,\xr',m')}(u)
        \,du
        ~\leq~
        \sum_{\ell\geq0} \ell \left|p_\ell(t,\xr,m) - p_\ell(t,\xr',m')\right|.
    \end{equation*}    
\end{lemma}

\begin{proof}
    Fix $(t,\xr,\xr',m,m')\in [0,T]\x\Cc^d\x\Cc^d\x\Pc_1(\Dc_E)\x\Pc_1(\Dc_E)$ and consider the two probability measures $P :=(p_\ell(t,\xr,m))_{\ell\in\N}$ and $P' :=(p_\ell(t,\xr',m'))_{\ell\in\N},$ which belong to $\Pc_1(\N)$ in view of Assumption \ref{A.1}.  Denote the cumulative distribution functions of $P$ and $P'$ as $F$ and $F'$ respectively, and observe that the generalized inverses are given as follows:  for all $u\in[0,1],$
    \begin{align*}
        F^{-1}(u)
        =
        \sum_{\ell\geq0}\ell \1_{I_\ell}(u)
        ~~\text{and}~~
        (F')^{-1}(u)
        =
        \sum_{\ell\geq0}\ell\1_{I'_{\ell}}(u),
    \end{align*}
    where $I_\ell:=I_\ell(t,\xr,m)$ and $I'_{\ell}:=I_{\ell}(t,\xr',m')$ for $\ell\in\N$.    
    Using the well-known representation of Wasserstein metric for one--dimensional distributions, it holds
    \begin{align*}
        \Wc_1(P,P')
        =
        \int_0^1
            \left|F^{-1}(u)-(F')^{-1}(u)\right|
        du
        =
        \int_0^1
            \sum_{\ell,\ell'\geq0}\left|\ell-\ell'\right|\1_{I_\ell\cap I'_{\ell'}}(u)
        du,
    \end{align*} 
    In addition, by change of variable, we also have that
    \begin{align*}
        \Wc_1(P,P')
        =
        \sum_{j\geq0}\left|F(j)-F'(j)\right|
        & =
        \sum_{j\geq0}
        \left|
            \sum_{\ell\leq j}p_\ell(t,\xr,m) - \sum_{\ell\leq j}p_\ell(t,\xr',m')
        \right|\\
        & \leq 
        \sum_{j\geq0}\sum_{\ell>j}\left|p_\ell(t,\xr,m)-p_\ell(t,\xr',m')\right|\\
        & = 
        \sum_{\ell\geq 0}\ell\left|p_\ell(t,\xr,m)-p_\ell(t,\xr',m')\right|.
    \end{align*}
    The conclusion follows immediately.
\end{proof}

\subsection{Proof of Theorem \ref{thm:strong_sol}}
\label{sec:proof_strong_sol}

\begin{proof}[Proof of Theorem \ref{thm:strong_sol}]
    We outline the strategy of the proof. We first establish existence and uniqueness of SDE~\eqref{eq:SDE_MKVB}--\eqref{eq:SDE_MKVB2} on a small interval $[0,T']$ by constructing a contraction mapping.  Then we shift the time window to $[T',2T']$ and so forth to establish the global well-posedness.\\

    \textit{Step 1.}
    To establish existence and uniqueness on a small interval, we construct a mapping $\Psi_t:\Pc_1(\Dc_E)\rightarrow\Pc_1(\Dc_E)$ for each $t\in[0,T],$ to which we aim to apply a Picard fixed point iteration. 
    Fix $\mu\in \Pc_1(\Dc_E)$ and denote by $Z$ the solution to the branching diffusion SDE~\eqref{eq:branch_SDE_classical} which is well-defined by Lemma \ref{lemma unique existence fixed flow of law}. 
    Then we define the map $\Psi_t$ on $\Pc_1(\Dc_E)$ onto itself as follows:
    \begin{align}\label{eq:psi}
        \Psi_t: \mu\in\Pc_1(\Dc_E) \mapsto \Lc(Z_{t\wedge\cdot}).
    \end{align}
     Notice that $\Psi_t(\mu)\in\Pc_1(\Dc_E)$  by Lemma \ref{lemma integrability} in view of~\eqref{eq:p1}. Therefore, existence and uniqueness of solution to SDE \eqref{eq:SDE_MKVB}--\eqref{eq:SDE_MKVB2} is equivalent to existence and uniqueness of a fixed point $\mu = \Psi_T(\mu).$ Actually we will rather consider the localized version 
    \begin{align*}
        \mu_t = \Psi_t(\mu_t),\quad\text{where  $\mu_t:=\mu\circ(\pi_t)^{-1}\in\Pc_1(\Dc_E).$} 
    \end{align*}       

\textit{Step 2.}
In this step, we prove that there exists $T'\in[0,T]$ such that the equation
\begin{align*}
    \mu_{T'} = \Psi_{T'}(\mu_{T'})
\end{align*}
admits a unique solution. 
To this end, we aim to show that $\Psi_{T'}$ is a contraction mapping when $T'$ is sufficiently small. Let us arbitrarily pick $\mu^1_{T'},\mu^2_{T'}\in\Pc_1(\Dc_E).$ Denote by $Z^i$ the unique solution to SDE~\eqref{eq:branch_SDE_classical} given $\mu^i_{T'}$ for $i=1,2$. 
Notice first that
\begin{align*}
    \Wc_1\left(\Psi_{T'}(\mu^1_{T'}),\Psi_{T'}(\mu^2_{T'})\right)
    \leq
    \E\left[
        d(Z^1_{T'\wedge\cdot},Z^2_{T'\wedge\cdot})
    \right]
    =
    \E\left[
        d_{T'}(Z^1,Z^2)
    \right].
\end{align*}
In view of Lemma \ref{lemma compare of two solutions}, there exist constants $c_d,c_w>0$ such that for any $T'>0$ small enough to have $1-c_d(T'+\sqrt{T'})>0,$ it holds
\begin{equation*}
    \E\left[
        d_{T'}(Z^1,Z^2)
    \right]
    \leq 
    \frac{c_w(T'+\sqrt{T'})}{1-c_d(T'+\sqrt{T'})}\E\left[\#K_0\right]e^{\gammab M T}
    \Wc_1\left(\mu^1_{T'},\mu^2_{T'}\right),
\end{equation*}
where we also used the trivial inequality $e^{\gammab M T'}\le e^{\gammab M T}.$
W.l.o.g.  we can take $T'$ even smaller to ensure that
\begin{align*}
    T'+\sqrt{T'}
    <
    \frac{1}{c_d+c_w\E\left[\#K_0\right]e^{\gammab M T}},
\end{align*}
so that 
\begin{equation}\label{eq:kappa}
    0
    <
    \kappa 
    :=
    \frac{c_w(T'+\sqrt{T'})}{1-c_d(T'+\sqrt{T'})}\E\left[\#K_0\right]e^{\gammab M T}
    <
    1.
\end{equation}
It follows that 
\begin{align*}
    \Wc_1(\Psi_{T'}(\mu^1_{T'}),\Psi_{T'}(\mu^2_{T'}))
    \leq
    \kappa \Wc_1(\mu^1_{T'},\mu^2_{T'}).
\end{align*}
In other words,  $\Psi_{T'}$ is a contraction mapping on $(\Pc_1(\Dc_E),\Wc_1)$. In particular, uniqueness of the solution to SDE~\eqref{eq:SDE_MKVB}--\eqref{eq:SDE_MKVB2} on $[0,T']$ follows immediately. Regarding existence, it suffices to observe that the sequence of iteration $(\Psi^{(n)}_{T'}(\mu_{T'}))_{n\in\N}$ converges as  a Cauchy sequence in $(\Pc_1(\Dc_E),\Wc_1)$ which is complete since $\Dc_E$ equipped with the uniform topology is complete (though not separable). 

\textit{Step 3.}
After establishing existence and uniqueness on $[0,T']$ in \textit{Step 2}, we now aim to extend these properties to a larger interval $[0,T'+\delta T]$ for some $\delta T>0$ to be determined. Denote $T'':=T'+\delta T$ with $\delta T\le T-T'.$ Fix $\mu^1_{T''},\mu^2_{T''}\in\Pc_1(\Dc_E)$ such that $\mu^1_{T'}=\mu^2_{T'}=\mu_{T'}$ where $\mu_{T'}$ is the unique solution to $\Psi_{T'}(\mu_{T'})=\mu_{T'}$ constructed in \textit{Step 2}. 
Let us denote by $Z^i\in\Dc_E$ the solution to \eqref{eq:branch_SDE_classical} given $\mu^i_{T''}$ for $i=1,2.$
In particular, it holds that $Z^1_{T'\wedge\cdot} = Z^2_{T'\wedge\cdot}$ by uniqueness in Lemma~\ref{lemma unique existence fixed flow of law}.
It follows that 
\begin{align*}
    \Wc_1\left(\Psi_{T''}(\mu^1_{T''}),\Psi_{T''}(\mu^2_{T''})\right)
    \leq
    \E\left[
        d(Z^1_{T''\wedge\cdot},Z^2_{T''\wedge\cdot})
    \right]
    =
    \E\left[
        \sup_{T'\leq t\leq T''}d_E(Z^1_t,Z^2_t)
    \right].
\end{align*}
By a straightforward extension of Lemma \ref{lemma compare of two solutions}, shifting the origin of time from $0$ to $T',$ we deduce that for any $\delta T\in (0,T-T']$ such that $1-c_d(\delta T+\sqrt{\delta T})>0,$
\begin{align*}
    \Wc_1\left(\Psi_{T''}(\mu^1_{T''}),\Psi_{T''}(\mu^2_{T''})\right)
    \leq &
    \frac{c_w(\delta T+\sqrt{\delta T})}{1-c_d(\delta T+\sqrt{\delta T})} \E\left[\#K_{T'}\right]e^{\gammab M \delta T}
        \Wc_1\left(\mu^1_{T''},\mu^2_{T''}\right).
\end{align*}
where $\# K_{T'}=\langle Z^1_{T'}, \1{}\rangle = \langle Z^2_{T'}, \1{}\rangle.$ In addition,  Lemma~\ref{lemma integrability} ensures that 
\begin{equation*}
    \E\left[\#K_{T'}\right]
    \leq
    \E\left[\#K_0\right] e^{\gammab M T'}.
\end{equation*}
It follows that
\begin{equation*}
    \Wc_1\left(\Psi_{T''}(\mu^1_{T''}),\Psi_{T''}(\mu^2_{T''})\right)
    \leq 
    \frac{c_w(\delta T+\sqrt{\delta T})}{1-c_d(\delta T+\sqrt{\delta T})}\E\left[\#K_0\right]
    e^{\gammab M T}\Wc_1(\mu^1_{T''},\mu^2_{T''}).
\end{equation*}
where we used the trivial inequality $e^{\gammab M T''}\le e^{\gammab M T}.$
Thus we can repeat the calculation from \textit{Step~2} and  choose $\delta T = \min(T',T-T')$ so that
\begin{align*}
    \Wc_1(\Psi_{T''}(\mu^1_{T''}),\Psi_{T''}(\mu^2_{T''}))
    \leq
    \kappa
    \Wc_1(\mu^1_{T''},\mu^2_{T''}),
\end{align*}
where $0<\kappa<1$ is given by~\eqref{eq:kappa}. 
It follows that SDE~\eqref{eq:SDE_MKVB}--\eqref{eq:SDE_MKVB2} admits a unique solution on $[0,T'+\delta T]$. 
We conclude the proof by repeating the above procedure finitely many times in order to cover the whole interval $[0,T].$
\end{proof}

\begin{remark}
\label{remark reason for L 1}
	We use the classical contraction argument to obtain existence and uniqueness of McKean--Vlasov branching diffusion. 
	However, unlike most of the literature on McKean--Vlasov diffusion, we work with the one--Wasserstein metric $\Wc_1$ instead of $\Wc_2$ and it appears that our argument fails otherwise.
To illustrate the main difficulty,  let us consider the simple case 
when $b = 0,$ $\sigma = 0,$ $p_2 = 1$ and $\gamma:[0,T]\x \Pc_2(\Dc_E)\to\R_+$ is Lipschitz in the second argument, \ie, there exists $L>0$ such that for all $(t,m,m')\in[0,T]\x\Pc_2(\Dc_E)\x\Pc_2(\Dc_E),$
	\begin{align*}
		|\gamma(t,m)-\gamma(t,m')|
		\leq
		L \Wc_2(m_t,m_t').
	\end{align*}
		The key of the contraction argument is to obtain an estimate similar to Lemma~\ref{lemma compare of two solutions}. 
	Using the same notations, we observe first that 	
	\begin{align*}
		d(Z^1,Z^2)
		=
		\sup_{0\leq t\leq T} \#K^{1\triangle 2}_t.
	\end{align*}
	Following \textit{Step 1} in the proof of Lemma \ref{lemma compare of two solutions}, we obtain from~\eqref{L} that
	   \begin{equation*}
        \E\left[d(Z^1,Z^2)^2\right]
        \leq   2 \E\left[\left(J^1_T + J_T^2\right)^2\right] + 2 \E\left[\left(J^3_T + J_T^4\right)^2\right].
    \end{equation*}
    To estimate both terms on the r.h.s.\  above, we use the following inequality: Denoting $\bar{Q}^k(ds,dz) := Q^k(ds,dz)  - ds \,dz$ the compensated Poisson measure, it holds under appropriate integrability condition,
    \begin{multline*}
    \E\left[
           \left( \int_{(0,T]\x\R_+^2}
                \sum_{k\in \K} f^k_t(z)
           \, \Qk(dt,dz) \right)^2
        \right] \\
        \begin{aligned}
         & \leq 2 \E\left[
           \left( \int_{(0,T]\x\R_+^2}
                \sum_{k\in \K} f^k_t(z)
            \bar{Q}^k(dt,dz) \right)^2
        \right] 
        + 
        2T\E\left[
          \int_0^T
                 \left( \sum_{k\in \K} \int_{\R_+^2}f^k_t(z) \,dz\right)^2
            dt 
        \right] \\
       & = 2 \E\left[
          \int_0^T
                \sum_{k\in \K} \int_{\R_+^2} \left( f^k_t(z)\right)^2 dz
            \,dt 
        \right]
        + 
        2T\E\left[
          \int_0^T
                 \left( \sum_{k\in \K} \int_{\R_+^2}f^k_t(z) \,dz\right)^2
            dt 
        \right].
       \end{aligned}
    \end{multline*}
    Using the inequality above with $f^k_t(z) = 2 \1_{k\in K^1_{t-}\setminus K^2_{t-}}  \1_{[0,1]\x[0,\gamma^1_t]}(z) + 2 \1_{k\in K^2_{t-}\setminus K^1_{t-}}  \1_{[0,1]\x[0,\gamma^2_t]}(z),$ it follows that 
    \begin{align*}
    \E\left[\left(J^1_T + J_T^2\right)^2\right]
    	 & \le 2 \E\left[
           \int_0^T
                \sum_{k\in K^{1\triangle 2}_t} 4 \gammab
            \,dt
        \right] 
        + 
        2T\E\left[
          \int_0^T
                 \left(\sum_{k\in K^{1\triangle 2}_t} 2 \gammab \right)^2
            dt 
        \right] \\
       & \le C
           \int_0^T
                \E\left[\left(\# K^{1\triangle 2}_t\right)^2\right]  
           \,dt.
    \end{align*}
    Similarly,  using $f^k_t(z) = 3 \1_{k\in K^{1\cap 2}_{t-}}\1_{[0,1]\x [\gamma^1_t\wedge\gamma^2_t, \gamma^1_t\vee\gamma^2_t]} (z),$ it holds
    \begin{align*}
    \E\left[\left(J^3_T + J_T^4\right)^2\right]
    	 & \le 2 \E\left[
           \int_0^T
                \sum_{k\in K^{1\cap 2}_t} 9 \left|\gamma^1_t-\gamma^2_t\right|
            \,dt
        \right] 
        + 
        2T\E\left[
          \int_0^T
                 \left(\sum_{k\in K^{1\cap 2}_t} 3 \left|\gamma^1_t-\gamma^2_t\right| \right)^2
            dt 
        \right] \\
        &  \le C 
           \int_0^T
                \left(\Wc_2(\mu^1_t,\mu^2_t) 
                + \Wc_2(\mu^1_t,\mu^2_t)^2
               \right)  
           \,dt.
    \end{align*}
    Combining the estimates above, we deduce that 
	   \begin{equation*}
        \E\left[d_T(Z^1,Z^2)^2\right]
        \leq   C  \int_0^T\E\left[
            d_t(Z^1,Z^2)^2
        \right] +  C  \int_0^T 
                \left(\Wc_2(\mu^1_t,\mu^2_t) 
                + \Wc_2(\mu^1_t,\mu^2_t)^2 
               \right)  
           \,dt.
    \end{equation*}
    Using Gr\"owall Lemma, it follows that
	   \begin{equation}\label{inequality remark final estimation}
        \E\left[d_T(Z^1,Z^2)^2\right]
        \leq   C \int_0^T
                \left(\Wc_2(\mu^1_t,\mu^2_t) 
                + \Wc_2(\mu^1_t,\mu^2_t)^2 
               \right)  
           \,dt.
    \end{equation}
    This inequality aims to play the role of Lemma~\ref{lemma compare of two solutions} to show that $\Psi_T$ defined by~\eqref{eq:psi} is a contraction in $(\Pc_2(\Dc_E),\Wc_2).$ 
    The problem comes from the fact that, although the l.h.s. of~\eqref{inequality remark final estimation} dominates $\Wc_2(\Psi_T(\mu^1),\Psi_T(\mu^2))^2$, the r.h.s.\  of~\eqref{inequality remark final estimation} contains a term
	$\Wc_2(\mu^1_t, \mu^2_t)$ which is not dominated by $\Wc_2(\mu^1_t, \mu^2_t)^2$.
	Therefore,  the classical argument using Gr\" onwall Lemma to show that $\Psi_T$ is a contraction fails in this setting.
\end{remark}

\section{Weak Formulation and Propagation of Chaos}
\label{sec:proof_weak}

	We provide here the proof of the propagation of chaos property established in Theorem~\ref{thm:weak_existence_limit}, which induces existence of weak solutions to the McKean--Vlasov branching diffusion SDE.
	The main idea is to consider the convergence of the martingale problem associated to the $n$-interacting branching diffusion toward the martingale problem associated to the  McKean--Vlasov branching diffusion.
	Similar techniques have been used by Lacker \cite{Lacker 2017} and Djete, Possama\"i and Tan \cite{Djete Possamai Tan 2022 ii} in the context of optimal control of McKean--Vlasov diffusion,
	and by He \cite{He 2023}  for the corresponding optimal stopping problem.

\subsection{Martingale Problem Formulation}

	As in the classical SDE theory, the notion of weak solution can be equivalently formulated by means of an appropriate martingale problem defined on the canonical space.

\paragraph{Martingale problem for $n$-interacting branching diffusion}
	Let us first consider the $n$-interacting branching diffusion process given in Definition \ref{def:weak_solution_n}.
	We introduce the corresponding canonical space
	$$
		\Omb_n ~:=~ \Dc_E^n,
	$$
	equipped with the canonical process $(Z^{1}, \cdots, Z^{n})$ and the canonical filtration $\Fbb^n = (\Fcb^n_t)_{t \in [0,T]}$ defined by $\Fcb^n_t := \sigma(Z^1_s, \cdots, Z^n_s, ~s \le t)$.
	Let us also consider the empirical measure
	$$
		\mu^n_t ~:=~ \frac1n \sum_{i=1}^n \delta_{Z^i_{t \wedge \cdot}}.
	$$
	Further, given $\Phi^n \in C^2_b(\R^n,\R)$ and $\bar{\varphi} = (\varphi_1, \cdots, \varphi_n)$ where $\varphi_i = (\varphi_i^k)_{k\in \K} \in C^2_b(\K\x\R^d,\R)$ for each $i =1, \cdots, n$,
	we define $\Phi^n_{\bar{\varphi}}: E^n \longrightarrow \R$ by
	\begin{align*}
		\Phi^n_{\bar{\varphi}}(e_1, \cdots, e_n) 
		~:=~
		\Phi^n(\<e_1,\varphi_1\>, \cdots, \< e_n, \varphi_n \>),
	\end{align*}
	as well as, for all $t \in [0,T]$, $i=1,\cdots, n$, $\omb = (\om^1, \cdots, \om^n) \in \Omb^n$ with $\om^i_t = \sum_{k \in K^i_t} \delta_{(k,\xr^{i,k}_t)}$,
	\begin{multline*}
		\Hc^n_i \Phi^n_{\bar{\varphi}} (t, \omb)
		:=
		\frac{1}{2} \partial^2_{i i} \Phi^n_{\bar{\varphi}}(\omb_t) \sum_{k\in K^i_t} \left|D \varphi^k_i (\xr^{i,k}_t) \sigma(t, \xr^{i,k}_{t\wedge\cdot}, m^n_t)\right|^2
		+
		\partial_{i} \Phi^n_{\bar{\varphi}} (\omb_t) \sum_{k\in K^i_t} \Lc \varphi^k_i(t, \xr^{i,k}_{t\wedge \cdot}, m^n_t) \\
		+
		\sum_{k\in K^i_t }\gamma(t, \xr^{i,k}_{t\wedge\cdot}, m^n_t)
		\bigg(
		\sum_{\ell \ge 0} \Phi^n_{\bar{\varphi}} \Big(\om^1_t, \cdots, \om^i_{t-} -\delta_{(k, \xr^{i,k}_{t-})}+\sum_{j = 1}^{\ell} \delta_{(kj, \xr^{i,k}_{t-})} , \cdots,\om^n_t\Big) p_{\ell} (t, \xr^{i,k}_{t\wedge\cdot}, m^n_t)-\Phi^n_{\bar{\varphi}}(\omb_{t-})
		\bigg),
	\end{multline*}
	where $m^n_t := \frac1n \sum_{j=1}^n \delta_{\om^j_{t\wedge\cdot}}$ and $\partial_{i} \Phi^n$ (resp. $\partial^2_{ii} \Phi^n$) respresents the first (resp. second) partial derivative of $\Phi^n$ w.r.t.\ the $i$-th coordinate.
	Let us also define the process $\Mc^{\Phi^n_{\bar{\varphi}}} = ( \Mc^{\Phi^n_{\bar{\varphi}}}_t)_{t \in [0,T]}$ on $\Omb_n$ by
	\begin{align*}
		\Mc^{\Phi^n_{\bar{\varphi}}}_t 
		:=
		\Phi^n_{\bar{\varphi}} \big( Z^1_t, \cdots, Z^n_t \big) 
		-
		\int_0^t \sum_{i=1}^n  \Hc^n_i \Phi^n_{\bar{\varphi}} \big(s, Z^1_{s\wedge\cdot}, \cdots, Z^n_{s\wedge\cdot} \big) ds.
	\end{align*}
	
	\begin{definition} \label{def:n_tree_mgl_pb}
		A solution to the martingale problem for $n$-interacting branching diffusion with initial condition $m_0 \in \Pc(E)$
		is a probability measure $\Pb^n$ on $\Omb_n$ satisfying 
		\begin{itemize}
		 \item[\rm (i)] $\Pb^n \circ (Z^1_0, \cdots, Z^n_0)^{-1} = m_0 \x \cdots \x m_0$
		 \item[\rm (ii)] $\Mc^{\Phi^n_{\bar{\varphi}}}$ is a $(\Pb^n, \Fbb^n)$--martingale for all 
		$\Phi^n \in C^2_b(\R^n,\R)$ and $\bar{\varphi} = (\varphi_1, \cdots, \varphi_n)$ such that $\varphi_i \in C^2_b(\K\x\R^d,\R)$ for each $i =1, \cdots, n$.
		\end{itemize}
	\end{definition}

	\begin{proposition} \label{prop:equiv_mgtpb_n}
		Let Assumption \ref{A.1} hold and $m_0\in\Pc_1(E).$ Then the following assertions hold:
		\begin{itemize}
		 \item[\rm (i)] Let $\alpha_n = \left(\Om^n, \Fc^n, \F^n, \P^n,  (Z^{i})_{i = 1, \cdots, n}, (W^{i,k}, Q^{i,k})_{k\in\K, i = 1, \cdots, n} \right)$ 
		be a weak solution in the sense of Definition \ref{def:weak_solution_n},
		then the probability $\Pb^n := \P^n \circ (Z^{1}, \cdots, Z^{n})^{-1}$ solves the corresponding martingale problem for $n$-interacting branching diffusion.
		 
		 \item[\rm (ii)]  Let $\Pb^n$ be a solution to the martingale problem for $n$-interacting branching diffusion, 
		then there exist, in a possibly enlarged space, Brownian motions and Poisson random measures,
		which, together with the canonical space $(\Omb_n, \Fcb^n_T, \Fbb^n, \Pb^n)$, give a weak solution in the sense of Definition \ref{def:weak_solution_n}.
		 
		 \item[\rm (iii)]  If we assume further that Assumption \ref{A.3} holds,  then there exists a solution to the martingale problem for $n$-interacting branching diffusion in the sense of Definition \ref{def:n_tree_mgl_pb}.
		\end{itemize}
	\end{proposition}
	
	\begin{proof} The arguments are rather classical. Let us provide a sketch of proof for completeness.
		First, to construct a solution to the martingale problem from a weak solution,
it suffices to apply It\^o's formula as in Claisse \cite[Proposition 3.2]{Claisse 2018}.
		Next, given a solution to the martingale problem $\Pb^n$,  it is also classical to construct, in a possibly enlarged space,  Brownian motions and Poisson random measures to represent the process by stochastic integrations, see, \eg, Kurtz \cite[Theorem 2.3]{Kurtz 2010}. This provides a weak solution.
		Finally, notice that the $n$-interacting branching diffusion solution to  \eqref{eq:n_tree_SDE} is actually a standard branching diffusion process, in the sense that it does not exhibit nonlinearity in the sense of McKean.
		As a consequence, we can apply classical weak convergence techniques as in {Strook and Varadhan \cite[Chapter 6.1]{Strook Varadhan 1997}} to deduce existence of a weak solution provided that the coefficients are bounded continuous.
	\end{proof}

\paragraph{Martingale problem for McKean--Vlasov branching diffusion}

 Let us  now consider McKean--Vlasov branching diffusion in the sense of Definition \ref{def:weak_solution}.  In this setting, we introduce two canonical spaces
	\begin{equation}
		\Omt ~:=~ \Dc_E
		~~\mbox{and}~
		\Omb
		~:=~
		\Dc_E \x \Pc(\Dc_E).
	\end{equation}
	Let us denote by $\Zt$ the canonical process on $\Omt$ and by $\Ft = (\Fct_t)_{t \in [0,T]},$ $\tilde{\Fc}_t = \sigma(\Zt_s,\, s \le t),$ the corresponding canonical filtration.
	Similarly, the canonical process on $\Omb$ is denoted by $(\Zb, \mub)$ and the  corresponding canonical filtration $\Fbb = (\Fcb_t)_{t \in [0,T]}$ is defined by $\Fcb_t := \sigma(\Zb_s, \mub_s,\, s \le t)$.
	Further, given $\Phi\in C^2_b(\R,\R)$ and $\varphi = (\varphi^k)_{k\in \K} \in C^2_b(\K\x\R^d,\R)$,
	let us define 
	\begin{align*}
		\Phi_{\varphi}(e) := \Phi(\<e,\varphi\>) = \Phi \Big(\sum_{k\in K}\varphi^k(x^k) \Big),
		~\mbox{for all}~
		e:=\sum_{k\in K}\delta_{(k,x^k)}\in E,
	\end{align*}
	as well as, for $\omt \in \Omt$ with $\omt_t = \sum_{k \in K_t} \delta_{(k, \xr^k_t)}$ and $\tilde{m} \in \Pc(\Omt)$,
	\begin{align*}
		&\Hc \Phi_{\varphi}(t, \omt, \tilde{m})
		:=
		\frac{1}{2}\Phi_{\varphi}''(\omt_t) \sum_{k\in K_t} \left|D\varphi^k(\xr^k_t) \sigma(t, \xr^k_{t\wedge\cdot}, \tilde{m}_t)\right|^2
		+
		\Phi_{\varphi}'(\omt_t) \sum_{k\in K_t} \Lc \varphi^k(t, \xr^k_{t \wedge \cdot}, \tilde{m}_t)\\
		&~~~ +
		\sum_{k\in K_t}\gamma(t, \xr^k_{t\wedge\cdot},\tilde{m}_t)
		\bigg(
			\sum_{\ell \ge 0} \Phi_{\varphi}\bigg( \omt_{t-} -\delta_{(k, \xr^k_{t-})}+\sum_{j = 1}^{\ell} \delta_{(kj, \xr^{k}_{t-} )} \bigg) p_{\ell} (t, \xr^k_{t\wedge\cdot},\tilde{m}_t)-\Phi_{\varphi}(\omt_{t-})
		\bigg).
	\end{align*}
	Let us also define the process $\tilde\Mc^{\Phi_\varphi, \tilde{m}} = (\tilde\Mc^{\Phi_\varphi, \tilde{m}}_t)_{t \in [0,T]}$ on $\Omt$ by
	\begin{align*}
		\tilde\Mc^{\Phi_\varphi, \tilde{m}}_t
		:=
		\Phi_{\varphi}(\Zt_t)-\int_0^t \Hc \Phi_{\varphi}(s, \Zt_{s \wedge \cdot}, \tilde{m}_s) ds.
	\end{align*}

	\begin{definition} \label{def:_MK_SDE_mgl_pb}
		A solution to the martingale problem for McKean--Vlasov branching diffusion with initial condition $m_0 \in \Pc(E)$ is a probability measure $\Pb$ on the canonical space $\Omb$ 
		such that 
		\begin{itemize}
			\item[\rm(i)] $\Pb \circ \Zb_0^{-1} = m_0$,
			
			\item[\rm(ii)] $\mub_t = \Lc^{\Pb} ( \Zb_{t\wedge \cdot} \,|\, \mub_t)$ for all $t \in [0,T]$,
		
			\item[\rm(iii)] for $\Pb$-a.e.\ $\omb \in \Omb$, 
			the process $\tilde\Mc^{\Phi_\varphi, \mub(\omb)}$ is a $(\mub(\omb), \Ft)$--martingale for all $\Phi\in C^2_b(\R,\R)$ and $\varphi = (\varphi^k)_{k\in \K} \in C^2_b(\K\x\R^d,\R)$.
		\end{itemize}
	\end{definition}

	\begin{proposition} \label{prop:equiv_martpb}
		Let Assumption \ref{A.1} hold and $m_0 \in \Pc_1(E).$ The following assertions hold:
		\begin{itemize}
		 \item[\rm(i)] Given a weak solution $\alpha = \left(\Om, \Fc, \F, \P, Z, \mu, (\Wk,\Qk)_{k\in\K} \right)$ to the McKean--Vlasov branching diffusion SDE in the sense of Definition \ref{def:weak_solution}, 
		the probability $\Pb := \P \circ (Z, \mu)^{-1}$ provides a solution to the corresponding martingale problem in the sense of Definition \ref{def:_MK_SDE_mgl_pb}.
		
		\item[\rm(ii)] Given a solution $\Pb$ to the martingale problem for McKean--Vlasov branching diffusion in the sense of Definition~\ref{def:_MK_SDE_mgl_pb},
		there exist, in a possibly enlarged space, Brownian motions and Poisson random measures,
		which, together with the canonical space $(\Omb, \Fcb_T, \Fbb, \Pb)$, give a weak solution in the sense of Definition \ref{def:weak_solution}. 
		\end{itemize}
	\end{proposition}

    \begin{proof}
        The proof is similar to Djete, Tan and Possama\"i \cite[Lemma 4.4]{Djete Possamai Tan 2022 i}. 
        Let $\alpha$ be a weak solution defined as above and $\Pb:=\P\circ(Z,\mu)^{-1}$.  We aim to prove that $\Pb$ solves the martingale problem for McKean Vlasov branching diffusion. It is straightforward to check that $\Pb \circ \Zb^{-1}_0 = m_0$ and $\mub_t = \Lc^{\Pb}(\Zb_{t\wedge\cdot} \,|\, \mub_t)$.  Then, given $h:\Dc_E\to\R$ and $g:\Pc(\Dc_E)\to\R$ arbitrary bounded measurable maps,  we observe that, for any $0\le s\le t\le T,$
        \begin{align*}
            \E^{\P}\left[
                h(Z_{s\wedge\cdot})g(\mu)(\tilde\Mc^{\Phi_\varphi, \mu}_t-\tilde\Mc^{\Phi_\varphi, \mu}_s)(Z)
            \right]
            & =
            \E^{\Pb}\left[
                h(\Zb_{s\wedge\cdot})g(\mub)(\tilde\Mc^{\Phi_\varphi, \mub}_t-\tilde\Mc^{\Phi_\varphi, \mub}_s)(\Zb)
            \right]\\
            & =
            \E^{\Pb}\left[
                \E^{\Pb}\left[
                    h(\Zb_{s\wedge\cdot})(\tilde\Mc^{\Phi_\varphi, \mub}_t-\tilde\Mc^{\Phi_\varphi, \mub}_s)(\Zb)
                \,\Big|\, \mub\right]
                g(\mub)
            \right]\\
            & =
            \E^{\Pb}\left[
                \E^{\mub}\left[
                    h(\Zt_{s\wedge\cdot})(\tilde\Mc^{\Phi_\varphi, \mub}_t-\tilde\Mc^{\Phi_\varphi, \mub}_s)
                \right]
                g(\mub)
            \right].
        \end{align*}
        where the last equality comes from $\mub = \Lc^{\Pb}(\Zb \,|\, \mub).$ In addition, it follows easily from It\^o's formula that 
        $$\E^{\P}\left[h(Z_{s\wedge\cdot})g(\mu)(\tilde\Mc^{\Phi_\varphi, \mu}_t-\tilde\Mc^{\Phi_\varphi, \mu}_s)(Z)\right]=0.$$ 
        By the arbitrariness of $g,$ we deduce that
        \begin{align*}
            \E^{\mub(\omb)}\left[
                    h(\Zt_{s\wedge\cdot})(\tilde\Mc^{\Phi_\varphi, \mub(\omb)}_t-\tilde\Mc^{\Phi_\varphi, \mub(\omb)}_s)
            \right]
            =
            0, \quad \Pb-\text{a.s.}
        \end{align*}
        We conclude by arbitrariness of $h$ that $\tilde\Mc^{\Phi_\varphi,\mu(\omb)}$ is a $(\mub(\omb), \Ft)$-martingale for $\Pb$--a.e.\ $\omb \in \Omb$. 
        
        {Conversely, by similar arguments as in Proposition \ref{prop:equiv_mgtpb_n}, if we have a solution to the martingale problem,  we can construct,  in a possibly enlarged space, Brownian motions and Poisson random measures to obtain a weak solution to the  McKean--Vlasov branching diffusion SDE.}
    \end{proof}

\subsection{Tightness of $n$-interacting Branching Diffusion}

	Let $(\alpha_n)_{n \ge 1}$ be a sequence of $n$--interacting branching diffusion in the sense of Definition \ref{def:weak_solution_n}, \ie, $$\alpha_n = \left(\Om^n, \Fc^n, \F^n, \P^n,  (Z^{i})_{i = 1, \cdots, n}, (W^{i,k}, Q^{i,k})_{k\in\K, i = 1, \cdots, n} \right)$$ is a weak solution to SDE~\eqref{eq:branching_MKSDE_weak} with initial condition $m_0 \in \Pc(E).$
	We consider the corresponding sequence of probability measures $(\Pb^n)_{n \ge 1}$ on the canonical space $\Omb = \Dc_E \x \Pc(\Dc_E),$ given by
	\begin{equation} \label{eq:def_Pbn}
		\Pb^n 
		~:=~
		\frac1n \sum_{i=1}^n \P^n \circ (Z^{i}, \mu^n)^{-1},
		~~\mbox{where}~
		\mu^n := \frac1n \sum_{i=1}^n \delta_{Z^{i}}. 
	\end{equation}
	
	To this end, we provide first an estimation on the number of particles in each of the $n$ branching diffusion processes.
	
	\begin{lemma} \label{lemm:N_finite}
	   Let Assumption \ref{A.1} hold and $m_0\in\Pc_1(E).$
	   It holds
		\begin{align}
			\sup_{n\ge 1} \E^{\P^n}\left[
			\sup_{0\leq t\leq T}\left\{\# K^{1}_t\right\}
			\right]
			<
			\infty.
		\end{align}
	\end{lemma}
	
	\begin{proof}
	 By similar arguments as Lemma \ref{lemma integrability}, we have 
	 \begin{equation*}
        \E^{\P^n} \left[
            \sup_{0\leq t\leq T}\left\{\# K^{1}_t\right\}
        \right]
        \leq
        \E^{\P^n}\left[\# K^{1}_0\right] e^{\gammab M T}.
    \end{equation*}
    The conclusion then follows from the fact that 
$\E^{\P^n}[\# K^{1}_0] = \int_E \< e, \1\>\, m_0(de) <+\infty.$ 
	\end{proof}

	\begin{proposition} \label{prop:tightness}
		Let Assumptions \ref{A.1} and \ref{A.3} hold and $m_0\in\Pc_1(E).$ Then the sequence $(\Pb^n)_{n \ge 1}$ is tight in $\Pc(\Omb)$.
	\end{proposition}

	\begin{proof}
	In view of Sznitman \cite[Proposition 2.2(ii)]{Sznitman 1991}, the tightness of $(\Pb^n)_{n \ge 1}$ is equivalent to the tightness of $(\E^{\P^n}[ \mu^n ])_{n \ge 1},$ where
	for all $\varphi: \Dc_E\to \R$ measurable bounded,
	$$
		\big\< \E^{\P^n} [ \mu^n ], \varphi  \big\>
		:=
		\E^{\P^n} \big[ \< \mu^n, \varphi \> \big]
		=
		\frac1n \sum_{i=1}^n \E^{\P^n} \big[ \varphi(Z^{i}) \big]
		=
		\E^{\P^n} \big[ \varphi(Z^{1}) \big].
	$$
	Therefore, it suffices to prove the tightness of 
	$(\P^n \circ (Z^{1})^{-1})_{n \ge 1}$ in $\Pc(\Dc_E)$.

	Recall that the process $Z^{1}_t=\sum_{k\in K^{1}_t} \delta_{(k,X^{1,k}_t)}$ takes values in $E$ equipped with the distance $d_E$ defined in \eqref{eq:def_d_E}, and that we consider the associated Skorokhod topology on $\Dc_E.$
	We aim to apply Aldous' criterion of tightness (see, e.g. Billingsley \cite[Theorem 16.10]{Billingsley 2013}).
	Namely we have to prove the following assertions:
	\begin{itemize}
		\item For all $\varepsilon>0$, there exists $n_0\in\N$ and $K>0$ such that
		\begin{equation} \label{eq:Aldous1}
			\forall n \geq n_0,~ \P^n \left(\sup_{0\leq t\leq T} d_E(Z^{1}_t,e_0) \geq K\right)\leq \varepsilon,
		\end{equation}
		where $e_0$ is the null measure on $\K\x\R ^d.$
		
		\item For all $\varepsilon>0$, it holds that
		\begin{equation} \label{eq:Aldous2}
			\lim_{\delta\rightarrow0}\sup_n\sup_{\tau\in\Tc^n}\P^n \left(d_E(Z^{1}_\tau, Z^{1}_{(\tau+\delta)\wedge T})>\varepsilon\right) = 0,
		\end{equation}
		where $\Tc^n$ is the collection of all $[0,T]$--valued stopping times on $(\Om^n, \Fc^n, \F^n).$
	\end{itemize}

	The first condition \eqref{eq:Aldous1} follows directly from Lemma \ref{lemm:N_finite} and Markov's inequality by observing that $d_E(Z^{1}_t,e_0)=\# K^{1}_t.$ 
 Regarding the second condition~\eqref{eq:Aldous2}, given $\delta >0$ and $\tau\in\Tc^n,$ we denote $\tau':= (\tau + \delta)\wedge T$ and we observe that
	\begin{multline}\label{eq:aldous}
		d_E(Z^{1}_\tau,Z^{1}_{\tau'})
		\leq
		\sum_{k\in K^{1}_\tau}\left|X^{1,k}_{\tau'} - X^{1,k}_{\tau}\right|\wedge1\\
		 +
		\int_{(\tau,\tau']\x[0,\gammab] \x [0,1]}
		\sum_{k\in K^{1}_{s-}}
		\sum_{\ell \geq0}
		(\ell +1)\1_{I_{\ell} (s,X^{1,k}_{s\wedge\cdot},\mu^n_{s-})\x[0,\gamma(s,X^{1,k}_{s\wedge\cdot},\mu^n_{s-})]}(z)
		Q^{1,k}(ds,dz).
	\end{multline}
	The second term corresponds exactly to the number of particles in $K^{1}_\tau \triangle K^{1}_{\tau'}$ as it counts the total number of particles born or dead between $\tau$ and $\tau'.$ As for the first term, it provides an upperbound to the contribution of particles in the common set $K^{1}_\tau \cap K^{1}_{\tau'}$ to the distance $d_E$ by neglecting potential death of particles. Here there is slight abuse of notation as we need to extend the path of particles who have died before time $\tau'$ as solution to SDE~\eqref{eq:SDE}.
	
	We start by dealing with the first term on the rhs of ~\eqref{eq:aldous}.
    Notice that
    \begin{align*}
        &\E^{\P^n}\left[
            \sum_{k\in K^{1}_\tau}\left|
                X^{1,k}_{\tau'}-X^{1,k}_\tau
            \right|
        \right]
        \leq
        \E^{\P^n}\left[
            \sum_{k\in K^{1}_\tau}
            \left(\int_\tau^{\tau'}
                \left|b(s,X^{1,k}_{s\wedge\cdot},\mu^n_s)\right|
            ds
        +
                \left|
                    \int_\tau^{\tau'}
                        \sigma(s,X^{1,k}_{s\wedge\cdot},\mu^n_s)
                    dW^{1,k}_s
                \right|
            \right)
        \right].
    \end{align*}
    On the one hand, we have
    \begin{align*}
        \E^{\P^n}\left[
            \sum_{k\in K^{1}_\tau}\int_\tau^{\tau'}
                \left|b(s,X^{1,k}_{s\wedge\cdot},\mu^n_s)\right|
            ds
        \right]
        \leq
        \left\|b\right\|_\infty \delta\, \E^{\P^n}\left[
            \# K^{1}_\tau
        \right].
    \end{align*}
    On the other hand, it follows from Burkholder--Davis--Gundy inequality that
    \begin{align*}
        \E^{\P^n}\left[
            \sum_{k\in K^{1}_\tau}
                \left|
                    \int_\tau^{\tau'}
                        \sigma(s,X^{1,k}_{s\wedge\cdot},\mu^n_s)
                    dW^{1,k}_s
                \right|
        \right]
        & = 
        \sum_{k\in\K}\E^{\P^n}\left[
            \left|
                \int_\tau^{\tau'}
                    \1_{k\in K^{1}_\tau}\sigma(s,X^{1,k}_{s\wedge\cdot},\mu^n_s)
                dW^{1,k}_s
            \right|
        \right]\\
        \leq &
        \sum_{k\in\K} C_1 \E^{\P^n}\left[
            \left(
                \int_\tau^{\tau'}
                    \1_{k\in K^{1}_\tau}\sigma(s,X^{1,k}_{s\wedge\cdot},\mu^n_s)^2
                ds
            \right)^{1/2}
        \right]\\
        \leq &
         C_1 \left\|\sigma\right\|_\infty \sqrt{\delta}\, \E^{\P^n}\left[\# K^{1}_\tau\right].
    \end{align*}
    Combining both estimations above, we deduce that
    \begin{align*}
        \label{estimate aldous condition 2, first part}
        \E^{\P^n}\left[
            \sum_{k\in K^{1}_\tau}\left|
                X^{1,k}_{\tau'}-X^{1,k}_\tau
            \right|
        \right]
        \leq
        C\left(\delta+\sqrt{\delta}\right)\E^{\P^n}\left[\sup_{0\leq t \leq T}\# K^{1}_t\right].
    \end{align*}
    As for the second term on the rhs of~\eqref{eq:aldous}, a direct computation yields
    \begin{multline}
        \E^{\P^n}\left[
            \int_{(\tau,\tau']\x[0,\gammab] \x [0,1]}
                \sum_{k\in K^{1}_{s-}}
                    \sum_{\ell\geq0}
                        (\ell+1)\1_{I_\ell(s,X^{1,k}_{s\wedge\cdot},\mu^n_{s-})\x[0,\gamma(s,X^{1,k}_{s\wedge\cdot},\mu^n_{s-})]}(z)
            Q^{1,k}(ds,dz)
        \right]\nonumber\\
        \begin{aligned}
        &=
        \E^{\P^n}\left[
            \int_\tau^{\tau'}
                \sum_{k\in K^{1}_s}\gamma(s,X^{1,k}_{s\wedge\cdot},\mu^n_s)\sum_{\ell\geq0}(\ell+1)p_\ell(s,X^{1,k}_{s\wedge\cdot},\mu^n_s)
            ds
        \right]\nonumber\\
        & \leq  
        \gammab(M+1)\delta\,\E^{\P^n}\left[
            \sup_{0\leq t\leq T}\# K^{1}_t
        \right].
        \end{aligned}
    \end{multline}
   It follows that
    \begin{align*}
        \E^{\P^n}\left[
            d_E(Z^{1}_\tau,Z^{1}_{(\tau+\delta)\wedge T})
        \right]
        \leq 
        C\left(\delta + \sqrt{\delta}\right)\E^{\P^n}\left[\sup_{0\leq t\leq T} \# K^{1}_t\right].
    \end{align*}
    In light of Lemma \ref{lemm:N_finite},  we deduce that
	\begin{equation*}
		\lim_{\delta \rightarrow 0}\sup_n\sup_{\tau\in\Tc^n}\E^{\P^n}\left[
		d_E(Z^{1}_\tau,Z^{1}_{(\tau+\delta)\wedge T})
		\right]
		=
		0.
	\end{equation*}
	Therefore, the second condition \eqref{eq:Aldous2} holds by Markov's inequality.
	\end{proof}

\subsection{Proof of Theorem \ref{thm:weak_existence_limit}}\label{sec:proof_weak_sol}

	Recall that the tightness of $(\P^n \circ (\mu^n)^{-1})_{n \ge 1}$ is established in Proposition \ref{prop:tightness}.
	More precisely, it is proved that the sequence $\Pb^n := \frac1n \sum_{i=1}^n \P^n \circ (Z^{i}, \mu^n)^{-1}$ is tight in $\Pc(\Omb).$
	Let us consider a subsequence, still denoted by $(\Pb^{n})_{n \ge 1},$ converging weakly to $\Pb.$
	We aim to prove that $\Pb$ is a solution to the martingale problem for McKean--Vlasov branching diffusion in the sense of Definition \ref{def:_MK_SDE_mgl_pb}. The conclusion then follows from Proposition~\ref{prop:equiv_martpb}.
	
	\vspace{0.5em}
	
	\noindent $\mathrm{(i)}$ First, it is easy to check that $\Pb \circ \Zb_0^{-1} = m_0.$ Indeed, we have $\P^n \circ (Z^{i}_0)^{-1} = m_0$ for each $i = 1, \cdots, n,$ and so $\Pb^n \circ \Zb_0^{-1} = m_0.$ Then we can pass to the limit as $\Zb_0$ is continuous for the Skorokhod topology.
	
	\vspace{0.5em}
	
	\noindent $\mathrm{(ii)}$ Next, by the definition of $\Pb^{n}$, it is easy to see that $\Lc^{\Pb^{n}} (\Zb_{t\wedge \cdot} \,|\, \mub_t) = \mub_t$, $\Pb^{n}$-a.e. \ for all $t \in[0,T].$
	Indeed, it holds for all $\phi \in C_b( \Pc(\Omt), \R)$ and $\psi \in C_b(\Dc_E, \R),$
	$$
		\E^{\Pb^{n}} \big[ \psi(\Zb_{t\wedge \cdot}) \phi(\mub_t) \big] 
		= \frac{1}{n} \sum_{i=1}^{n} \E^{\P^{n}} \big[ \psi(Z^{i}_{t\wedge \cdot}) \phi(\mu^{n}_t) \big] 
		= \E^{\Pb^{n}} \big[  \<\mub_t, \psi \> \phi(\mub_t) \big].
	$$
	It follows by passing to the limit that
	$$
		\E^{\Pb} \big[ \psi(\Zb_{t\wedge \cdot}) \phi(\mub_t) \big] = \E^{\Pb} \big[  \<\mub_t, \psi \> \phi(\mub_t) \big].
	$$
	We conclude that $\Lc^{\Pb} (\Zb_{t\wedge \cdot} \,|\, \mub_t) = \mub_t$, $\Pb$-a.e.\ for all $t \in [0,T]$.
	
	\vspace{0.5em}
	
	\noindent $\mathrm{(iii)}$ It remains to prove that $\tilde\Mc^{\Phi_\varphi, \mub(\omb)}$ is a $(\mub(\omb), \Ft)$-martingale for $\Pb$--a.e.\ $\omb \in \Omb.$
	Equivalently, we aim to show that it holds for all $0\le s\le t \le T$ and $h:\Omt \to \R$ bounded continuous  $\Fct_s$--measurable,
	\begin{equation} \label{eq:Mct_Pb_martingale}
		\E^{\Pb} \left[
			{\< \mub, (\tilde\Mc^{\Phi_\varphi, \mub}_t-\tilde\Mc^{\Phi_\varphi, \mub}_s) h \>^2}
		\right]
		=
		0.
	\end{equation}
	We start by observing that $\tilde\Mc^{\Phi_\varphi, \mu^n}_t(Z^{i})$ is a martingale under $\P^n$ with representation
	\begin{multline*}
		\tilde\Mc^{\Phi_\varphi, \mu^n}_t(Z^{i})
		= 
		\Phi_{\varphi}(Z^{i}_0)
		 +
		\int_0^t
		\Phi_{\varphi}'(Z^{i}_s)\sum_{k\in K^{i}_s}D\varphi^k(X^{i, k}_s)\sigma(s,X^{i, k}_{s\wedge\cdot},\mu^n_s)
		dW^{i,k}_s\\
		 +
		\int_{(0,t]\x[0,\gammab] \x [0,1]}
		\sum_{k\in K^{i}_{s-}} G^{i,k}_s(z)
		\bar{Q}^{i,k}(ds,dz),
	\end{multline*}
	where $\bar{Q}^{i,k}(ds,dz):=Q^{i,k}(ds,dz) - ds\,dz$ is the compensated Poisson random measure and 
    \begin{align*}
        G^{i,k}_s(z)
        :=
        \sum_{\ell\geq0}\left(
            \Phi_{\varphi}(Z^{i}_{s-} + \sum_{j=1}^{\ell}\delta_{(kj,X^{i, k}_s)})-\Phi_{\varphi}(Z^{i}_{s-})
        \right)\1_{I_\ell(s,X^{i,k}_{s\wedge\cdot},\mu^n_{s-})\x[0,\gamma(s,X^{i,k}_{s\wedge\cdot},\mu^n_{s-})]}(z).
    \end{align*}
    Since $(W^{i,k},Q^{i,k})_{k\in\K, i=1\cdots,n}$ are independent Brownian motions and Poisson random measures, the martingales $(\tilde\Mc^{\Phi_\varphi, \mu^n}_t(Z^{i}))_{i=1,...,n}$ are orthogonal under $\P^n$. It follows that
    \begin{align*}
        \E^{\Pb^n}\left[
            \<\mub, (\tilde\Mc^{\Phi_\varphi, \mub}_t-\tilde\Mc^{\Phi_\varphi, \mub}_s)h\>^2
        \right]
        & =
        \E^{\P^n}\left[
            \left(
                \frac{1}{n}\sum_{i=1}^n h(Z^{i})\left(\tilde\Mc^{\Phi_\varphi, \mu^n}_t(Z^{i})-\tilde\Mc^{\Phi_\varphi, \mu^n}_s(Z^{i})\right)
            \right)^2
        \right]\\
        & = 
        \frac{1}{n^2}\sum_{i=1}^n\E^{\P^n}\left[
            h(Z^{i})^2\left(\tilde\Mc^{\Phi_\varphi, \mu^n}_t(Z^{i})-\tilde\Mc^{\Phi_\varphi, \mu^n}_s(Z^{i})\right)^2
        \right]\\
        & = 
        \frac{1}{n} \E^{\P^n}\left[
            h(Z^{1})^2\left(\tilde\Mc^{\Phi_\varphi, \mu^n}_t(Z^{1})-\tilde\Mc^{\Phi_\varphi, \mu^n}_s(Z^{1})\right)^2
        \right].
    \end{align*}
    In addition, the increment $\tilde\Mc^{\Phi_\varphi, \mu^n}_t(Z^{1})-\tilde\Mc^{\Phi_\varphi, \mu^n}_s(Z^{1})$ admits the following quadratic variation:
        \begin{equation*}
        \int_s^t
            \Phi_{\varphi}'(Z^{1}_r)^2\sum_{k\in K^{1}_r}\left|D\varphi^k(X^{1,k}_r)\sigma(r,X^{1,k}_{r\wedge\cdot},\mu^n_r)\right|^2
        \,dr
        +
        \int_s^t
            \int_{[0,\gammab] \x [0,1]} \sum_{k\in K^{1}_r}G^{1,k}_r(z)^2 dz
        \,dr.
    \end{equation*}
    By Assumption \ref{A.1}, together with the fact that $\Phi\in C^2_b(\R)$ and $\varphi \in C^2_b(\K\x\R^d)$, there exists a constant $C>0$ that may vary from line to line such that 
    \begin{align*}
        \int_s^t
            \Phi_{\varphi}'(Z^{1}_r)^2\sum_{k\in K^{1}_r}\left|D\varphi^k(X^{1,k}_r)\sigma(r,X^{1,k}_{r\wedge\cdot},\mu^n_r)\right|^2
        \,dr
        \leq
        C\int_s^t
            \# K^{1}_r
        \,dr,
    \end{align*}
    and
    \begin{multline*}
       \int_s^t
            \int_{[0,\gammab] \x [0,1]} \sum_{k\in K^{1}_r}G^{1,k}_r(z)^2 dz
        \,dr \\
        \begin{aligned}
        & =  \int_s^t\sum_{k\in K^{1}_r}\gamma(r,X^{1,k}_{r\wedge\cdot},\mu^n_r)
            \sum_{\ell\geq0}\Big(
                \Phi_{\varphi}(Z^{1}_r + \sum_{j=1}^{\ell}\delta_{(kj,X^{1,k}_r)})-\Phi_{\varphi}(Z^{1}_r)
            \Big)^2p_\ell(r,X^{1,k}_{r\wedge\cdot},\mu^n_r)
        \,dr \\
       &  \leq
        C\int_s^t
            \# K^{1}_r
        \,dr.
        \end{aligned}
    \end{multline*}
    Along with Lemma \ref{lemm:N_finite}, it follows that
    \begin{align*}
	\E^{\Pb^n}\left[
            \<\mu, (\tilde\Mc^{\Phi_\varphi, \mu}_t-\tilde\Mc^{\Phi_\varphi, \mu}_s)h\>^2
        \right]
        \leq
        \frac{C\|h\|^2(t-s)}{n}\E^{\P^n}\left[
            \sup_{s\leq r\leq t}\# K^{1}_r
        \right]
        \xrightarrow[n\rightarrow\infty]{}
        0.
    \end{align*}
    Finally,  following the arguments of the proof of~\cite[Lemma~4.14]{Claisse Ren Tan 2019}, we can prove, up to the localizing sequence $\tilde{\tau}_{\tilde{n}}:=\inf\{r\ge 0;\ \langle \Zt_r,\1\rangle\ge \tilde{n}\},$ that the following convergence holds:  for all $s, t\notin D_{\Pb},$ a countable subset of $[0,T],$ such that $s\le t,$
	$$
		\E^{\Pb^n}\left[
		\<\mub, (\tilde\Mc^{\Phi_\varphi, \mub}_{t\wedge \tilde{\tau}_{\tilde{n}}}-\tilde\Mc^{\Phi_\varphi, \mub}_{s\wedge \tilde{\tau}_{\tilde{n}}})h\>^2
		\right]
		\xrightarrow[n\rightarrow\infty]{}
		\E^{\Pb} \left[
			\< \mub, (\tilde\Mc^{\Phi_\varphi, \mub}_{t\wedge \tilde{\tau}_{\tilde{n}}}-\tilde\Mc^{\Phi_\varphi, \mub}_{s\wedge \tilde{\tau}_{\tilde{n}}}) h \>^2
		\right].
	$$
	Note that $\tilde{\tau}_{\tilde{n}}$ is continuous for the Skorokhod topology in view of Jacod and Shirayev~\cite[Proposition~VI.2.7]{Jacod 2003}. 
	This proves that for all $s, t \notin D_{\Pb},$ $\Pb$--a.e.  $\omb,$
		$$
	  \< \mub(\omb), (\tilde\Mc^{\Phi_\varphi, \mub(\omb)}_{t\wedge \tilde{\tau}_{\tilde{n}}}-\tilde\Mc^{\Phi_\varphi, \mub(\omb)}_{s\wedge \tilde{\tau}_{\tilde{n}}}) h \>
		= 0.
	$$	
		Using the continuity of $\tilde\Mc^{\Phi_\varphi, \mub(\omb)}$, we obtain that this equality actually holds for all $0 \le s \le t \le T,$ and thus the stopped process $\tilde\Mc^{\Phi_\varphi, \mub(\omb)}_{\tilde{\tau}_{\tilde{n}}\wedge\cdot}$ is a $(\mub(\omb), \Ft)$--martingale for $\Pb$--a.e.\ $\omb.$ 
		To conclude, it remains to observe that, by arguments similar to Lemma~\ref{lemm:N_finite}, it holds
	\begin{equation*}
	\E^{\mu(\omb)}\left[
			\sup_{0\leq t\leq T} {\langle \Zt_t,\1\rangle}
			\right]
			<
			\infty, \quad \Pb-\text{a.s.},
	\end{equation*}
	so that the process $\tilde\Mc^{\Phi_\varphi, \mub(\omb)}$ is uniformly integrable and thus it is a $(\mub(\omb), \Ft)$--martingale.
	\qed

\appendix
\section{Stability of McKean--Vlasov Branching Diffusion}\label{sec:appendix}

Let $\tilde{b},\tilde{\sigma},\tilde{\gamma},(\tilde{p}_\ell)_{\ell\geq0}$ be respectively drift and diffusion coefficients, death rate and progeny distribution satisfying Assumptions \ref{A.1} and \ref{A.2}. Let $\tilde{\xi}$ be an $E$--valued $\Fc_0$--random variable such that $\E[\langle\tilde{\xi}, \1{}\rangle]< +\infty.$ 
Denote by $\tilde{Z}_t:=\sum_{k\in \tilde{K}_t}\delta_{(k,\tilde{X}^k_t)}$ the unique solution to SDE~\eqref{eq:SDE_MKVB}--\eqref{eq:SDE_MKVB2} with coefficients $\tilde{b},\tilde{\sigma},\tilde{\gamma},(\tilde{p}_\ell)_{\ell\geq0}$ and initial condition $\tilde{\xi}$ instead of $b,\sigma,\gamma,(p_\ell)_{\ell\geq0}$ and $\xi.$

\begin{proposition}[Stability of SDE]
    \label{prop:stability}
    Let Assumptions \ref{A.1} and \ref{A.2} hold and assume that $\E[\langle \xi, \1{}\rangle]< +\infty.$   Let $Z$ be the unique solution to SDE \eqref{eq:SDE_MKVB}--\eqref{eq:SDE_MKVB2} and $\tilde{Z}$ be defined as above. Then there exists a constant $C>0$ such that
    \begin{align*}
        \E\left[d (Z,\tilde{Z})\right]
        \leq
        C\left(
            \E\left[d_E(\xi,\tilde{\xi})\right]
            +
            \|b-\tilde{b}\|
            +
            \|\sigma-\tilde{\sigma}\|
            +
            \|\gamma-\tilde{\gamma}\|
            +
             \Big\|\sum_{\ell \geq 0} \ell  |p_{\ell} - \tilde{p}_{\ell}| \Big\|
        \right).
    \end{align*}
\end{proposition}

\begin{proof}
   Most of the proof follows from estimates established in Lemma \ref{lemma compare of two solutions}.  Denote $K^{\triangle}_t := K_t\triangle \tilde{K}_t$ and $K^{\cap}_t = K_t\cap \tilde{K}_t$ and recall that
    \begin{align}\label{eq:stability}
        d_{T}(Z,\tilde{Z}) = \sup_{0\leq t\leq T} d_E(Z_t,\tilde{Z}_t) \quad \text{where } d_E(Z_t,\tilde{Z}_t) = \# K^{\triangle}_t + \sum_{k\in K^{\cap}_t}  \big|X^k_t-\tilde{X}^k_t\big|\wedge 1.
    \end{align}
    We observe first by using similar arguments as in \textit{Step 1} of the proof of Lemma \ref{lemma compare of two solutions}, namely, combining \eqref{L} with \eqref{J1J2}, \eqref{regularity J3J4} and \eqref{regularity J5}, that
    \begin{multline*}
        \E\left[
            \sup_{0\leq t\leq T}\#K^{\triangle}_t
        \right]
        \leq
        \E\left[\#K^{\triangle}_0\right]
        +
      \gammab M \E\left[ 
      \int_0^{T} \# K^{\triangle}_t \, dt
          \right]
          \\
        +
       \left(
        (M+1) \|\gamma-\tilde{\gamma}\|
            +
            \gammab \Big\|\sum_{\ell \geq 0} \ell  |p_{\ell} - \tilde{p}_{\ell}| \Big\|
        \right)
        \E\left[ 
      \int_0^{T} \# K^{\cap}_t \, dt
          \right].
    \end{multline*}
  In addition, it follows from similar arguments as in \textit{Step~2} of the proof of Lemma \ref{lemma compare of two solutions} that
    \begin{multline*}
        \E\left[
           \sup_{0\leq t\leq T} \sum_{k\in K^{\cap}_t} \big|X^k_t-\tilde{X}^k_t\big|\wedge 1
        \right]
        \leq 
        \E\left[\sum_{k\in K^{\cap}_0}\left|X^k_0-\tilde{X}^k_0\right|\wedge1\right] 
        +
        \|b-\tilde{b}\| \E\left[\int_0^{T}{\#K^{\cap}_t} \,dt\right]
        \\
        +
       C_1 \sqrt{T}  \|\sigma-\tilde{\sigma}\|
       \E\left[\sum_{k\in\K}\sup_{0\leq t\leq T}\1_{k\in K^{\cap}_t}\right]
        +       
        \gammab M  \E\left[\int_0^{T}\left({\#K^{\triangle}_t + \sum_{k\in K^\cap_t} \big\|X^k_{t\wedge \cdot} -\tilde{X}^k_{t\wedge \cdot} \big\|\wedge 1} \right)dt\right].
    \end{multline*}
Combining both inequalities above together with the following consequence of Lemma~\ref{lemma integrability}
    \begin{equation*}
    \E\left[
            \int_0^{T}
             \# K^{\cap}_t \, dt
        \right] 
        \le 
        T
            \E\left[
             \sum_{k\in\K} \sup_{0\leq t\leq T}\1_{k\in K^{\cap}_t} 
        \right] 
         \le 
            C_0 T e^{\gammab M T},
    \end{equation*}
    where $C_0 :=  \max(\E[\# K_0], \E[\# \tilde{K}_0]),$
we deduce that 
    \begin{multline*}
        \E\left[
           d_{T}(Z,\tilde{Z})
        \right]    
        \leq \E\left[d_E(Z_0, \tilde{Z}_0)\right] 
        + 
        2 \gammab M \int_0^{T} \E\left[d_t(Z,\tilde{Z})\right]  dt
           \\
            +
            \bigg(
           \|b-\tilde{b}\| 
            +
             C_1 T^{-\frac{1}{2}} \|\sigma-\tilde{\sigma}\|
            +
            (M+1) \|\gamma-\tilde{\gamma}\|
            +
            \gammab \Big\|\sum_{\ell \geq 0} \ell  |p_{\ell} - \tilde{p}_{\ell}| \Big\|
            \bigg)
          C_0 T e^{\gammab M T}.
    \end{multline*}
    The conclusion follows immediately by Gr\"owall Lemma.
\end{proof}

\end{document}